\documentclass[10pt]{amsart}
\pdfoutput=1
\calclayout
\usepackage[T1]{fontenc}
\usepackage[utf8]{inputenc}
\usepackage{amsmath,amsfonts,amssymb,amsthm}
\usepackage{mathtools}
\usepackage{tensor}
\usepackage{pgf,tikz}
\usetikzlibrary{matrix,arrows,calc}
\usepackage{tikz-cd}
\usepackage{shuffle}
\usepackage{xcolor}
\definecolor{darkgreen}{rgb}{0,0.5,0}
\definecolor{lightblue}{rgb}{0.4,0.4,0.8}
\usepackage[frozencache]{minted}
\usepackage[
colorlinks, 
	citecolor=lightblue,
        pdfauthor={Martin Lüdtke}
        pdftitle={Refined Chabauty--Kim computations for the thrice-punctured line over Z[1/6]}
]{hyperref}
\usepackage{cleveref}

\numberwithin{equation}{section}

\usepackage{enumitem}

\usepackage[colorinlistoftodos]{todonotes}

\newtheorem{thm}{Theorem}
\newtheorem{prop}[thm]{Proposition}
\newtheorem{lemma}[thm]{Lemma}
\newtheorem{cor}[thm]{Corollary}
\newtheorem{conjecture}[thm]{Conjecture}

\theoremstyle{remark}
\newtheorem{rem}[thm]{Remark}

\theoremstyle{definition}
\newtheorem{defn}[thm]{Definition}
\newtheorem{problem}[thm]{Problem}

\newcommand{\bA}{\mathbb A}

\newcommand{\bC}{\mathbb C}

\newcommand{\rd}{\mathrm d}
\newcommand{\bF}{\mathbb F}

\newcommand{\rH}{\mathrm H}

\newcommand{\MT}{\mathrm{MT}}
\newcommand{\bP}{\mathbb P}

\newcommand{\bQ}{\mathbb Q}
\newcommand{\bZ}{\mathbb Z}
\newcommand{\cO}{\mathcal{O}}

\newcommand{\et}{\operatorname{\acute{e}t}}

\DeclareMathOperator{\dR}{dR}

\DeclareMathOperator{\loc}{loc}
\DeclareMathOperator{\Li}{Li}

\DeclareMathOperator{\PL}{PL}

\DeclareMathOperator{\Sel}{Sel}
\DeclareMathOperator{\Spec}{Spec}

\usepackage[
	backend=biber,
	style=alphabetic,
	minalphanames=3,
	maxnames=99,
	maxalphanames=4,
	giveninits=true,
	isbn=false,
	doi=false,
]{biblatex}
\numberwithin{thm}{section}
\numberwithin{table}{section}

\title[Refined Chabauty--Kim computations for the thrice-punctured line]{Refined Chabauty--Kim computations for the thrice-punctured line over $\bZ[1/6]$}
\author{Martin Lüdtke}

\begin{document}
	
	\begin{abstract}
		The Chabauty--Kim method and its refined variant by Betts and Dogra aim to cut out the $S$-integral points $X(\bZ_S)$ on a curve inside the $p$-adic points $X(\bZ_p)$ by producing enough Coleman functions vanishing on them. We derive new functions in the case of the thrice-punctured line when $S$ contains two primes. We describe an algorithm for computing refined Chabauty--Kim loci and verify Kim's Conjecture over $\bZ[1/6]$ for all choices of auxiliary prime~$p < 10{,}000$.
	\end{abstract}
	
	\maketitle
	
	\setcounter{tocdepth}{1}
	\tableofcontents
	
	\thispagestyle{empty}
	
	\section{Introduction}
	\label{sec: intro}
	
	Let $S$ be a finite set of primes and let $X = \bP^1 \smallsetminus \{0,1,\infty\}$ be the thrice-punctured line over the ring of $S$-integers $\bZ_S$. By the Siegel--Mahler theorem, the set of $S$-integral points $X(\bZ_S)$ is finite.\footnote{The finiteness of $X(\bZ_S)$ is often referred to as Siegel's theorem but it was first proved by Mahler \cite[Folgerung~2]{mahler33} in 1933. He uses a generalisation of Siegel's proof from 1921 of the finiteness of the set of \emph{integral} points $X(\cO_K)$ over a general number field~$K$. Siegel's result is not stated explicitly but it can be deduced from \cite[Satz 7, Zusatz~1]{siegel1921} which implies that the polynomial $f(x) = x(1-x)$ represents a unit of~$K$ only finitely many times as $x$ runs through the ring of integers of~$K$. 
	See \cite[Summary]{evertse-gyory} for a historical overview.} Kim \cite{kim:motivic} gave a new $p$-adic proof of this fact by constructing, for any prime $p \not\in S$, a descending chain of subsets of $X(\bZ_p)$ containing $X(\bZ_S)$:
	\[ X(\bZ_p) \supseteq X(\bZ_p)_{S,1} \supseteq X(\bZ_p)_{S,2} \supseteq \ldots \quad \supseteq X(\bZ_S).  \]
	The set $X(\bZ_p)_{S,n}$ is called the Chabauty--Kim locus of depth~$n$. Kim showed that the sets $X(\bZ_p)_{S,n}$ eventually become finite, so that $X(\bZ_S)$ must be finite as well. This suggests the following strategy for computing $X(\bZ_S)$: find as many points in $X(\bZ_S)$ as possible; then compute $X(\bZ_p)_{S,n}$ for some $p \not\in S$ and $n \geq 1$. This gives a lower and an upper bound. If they match, we have found $X(\bZ_S)$. In order for this strategy to have a chance of succeeding we need that the set $X(\bZ_p)_{S,n}$ contains no $p$-adic points which are not $S$-integral points, at least for sufficiently large~$n$. In other words, the inclusion $X(\bZ_S) \subseteq X(\bZ_p)_{S,n}$ should eventually be an equality. This is the content of Kim's Conjecture \cite[Conj.~3.1 \& §8.1]{BDCKW}. The heuristic behind this is that as~$n$ grows larger, we find more and more independent Coleman functions vanishing on $X(\bZ_p)_{S,n}$, and any $p$-adic point in the intersection of their zero loci should be there for a good reason, namely being an $S$-integral point. 
	
	Computing the sets $X(\bZ_p)_{S,n}$ is difficult in practice and so far has only been achieved in cases where~$S$ contains at most one prime \cite{BDCKW,BKL:chabauty_kim_sc,DCW:mixedtate1,CDC:polylog1}. In this paper we focus on the \emph{refined} Chabauty--Kim sets $X(\bZ_p)_{S,n}^{\min}$ introduced by Betts and Dogra \cite{BD:refined}. These are potentially smaller than the sets $X(\bZ_p)_{S,n}$ but still contain $X(\bZ_S)$. It is natural to formulate Kim's Conjecture also for the refined sets:
	
	\begin{conjecture}[Refined Kim's Conjecture]
		\label{kim-conjecture-refined}
		$X(\bZ_p)_{S,n}^{\min} = X(\bZ_S)$ for $n \gg 0$.
	\end{conjecture}
	
	If Kim's Conjecture holds for the original unrefined sets then \Cref{kim-conjecture-refined} also holds. Recently, verifying \Cref{kim-conjecture-refined} (for a hyperbolic curve of any genus) has been proposed as a strategy for proving Grothendieck's Section Conjecture for locally geometric sections \cite[Theorem~A]{BKL:chabauty_kim_sc}. 
	In this paper we verify \Cref{kim-conjecture-refined} for the thrice-punctured line over $\bZ[1/6]$ for many choices of the auxiliary prime~$p$.
	
	\begin{thm}[= \Cref{thm:main}]
		\label{thm:main-kim-intro}
		\Cref{kim-conjecture-refined} for $S = \{2,3\}$ holds in depth~$n = 4$ for all primes~$p$ with $5 \leq p < 10{,}000$.
	\end{thm}
	
	Previously, Kim's Conjecture for $S = \{2,3\}$ (refined or unrefined) was not known to hold for any prime~$p$. For $S = \{2\}$, \Cref{kim-conjecture-refined} can be proved for all odd primes~$p$ by purely algebraic reasoning \cite[Theorem~B]{BKL:chabauty_kim_sc}. In contrast, our proof of \Cref{thm:main-kim-intro} uses a combination of theoretical results and computer calculations. 
	
	For the most part we work with $S = \{2,q\}$ for an arbitrary odd prime~$q$. The refined Chabauty--Kim method for such sets was first applied in depth~$n = 2$ in \cite{BBKLMQSX}. There, an equation is derived which (essentially) defines the depth~2 locus $X(\bZ_p)_{\{2,q\},2}^{\min}$ inside $X(\bZ_p)$. It takes the form
	\begin{equation}
		\label{eq:depth2-function-intro}
		\log(2) \log(q) \Li_2(z) - a_{\tau_q \tau_2} \log(z) \Li_1(z) = 0
	\end{equation}
	for some computable $p$-adic constant $a_{\tau_q \tau_2} \in \bQ_p$. Here, $\log$, $\Li_1$ and $\Li_2$ are $p$-adic (poly)logarithm functions. One part of this paper is devoted to systematically computing the depth~2 loci $X(\bZ_p)_{\{2,q\},2}^{\min}$ for any~$p$ and~$q$. In §\ref{sec:depth2-loci} we describe an algorithm to achieve this. A \textsf{SageMath} \cite{sagemath} implementation is available at \url{https://github.com/martinluedtke/RefinedCK}.
	Using this code, we computed the depth~2 loci for many combinations of~$p$ and~$q$. We present our findings in §\ref{sec: sizes} and explain the observed behaviour by analysing the Newton polygons of power series. In particular, we are able to explain why the Chabauty--Kim loci are exceptionally large for the auxiliary primes $p = 1093$ and $p = 3511$. This is related to the fact that these are Wieferich primes, i.e., primes for which $2^{p-1} \equiv 1 \bmod p^2$. 
	
	In a few cases, notably whenever $q \geq 5$ is a Mersenne prime or a Fermat prime and we take $p = 3$, \Cref{kim-conjecture-refined} for $S = \{2,q\}$ holds already in depth~$2$ \cite[Cor.~3.15]{BBKLMQSX}. Most of the time, however, Eq.~\eqref{eq:depth2-function-intro} has $p$-adic solutions which are not $S$-integral points. In this case, one has to go to higher depth in order to verify Kim's Conjecture. 
	We take $n = 4$, show that in order to verify \Cref{kim-conjecture-refined} it suffices to look at a certain subset $X(\bZ_p)_{\{2,q\},4}^{(1,0)}$ of the refined Chabauty--Kim locus (defined in §\ref{sec:refined}) and derive an equation which holds on this set:
	\begin{thm}
		\label{thm:main-function-intro}
		Let $S = \{2,q\}$ for some odd prime~$q$ and let $p \not \in S$ be an auxiliary prime. Then every point in the refined Chabauty--Kim locus $X(\bZ_p)_{\{2,q\},4}^{(1,0)}$ satisfies, in addition to Eq.~\eqref{eq:depth2-function-intro}, a nontrivial equation of the form
		\begin{equation}
			\label{eq:depth4-equation-intro}
			a \Li_4(z) + b \log(z) \Li_3(z) + c \log(z)^3 \Li_1(z) = 0
		\end{equation}
		for certain $p$-adic constants $a,b,c \in \bQ_p$.
	\end{thm}
	
	\Cref{thm:main-function-intro} is a simplified version of \Cref{thm: 10 equations} where more precise expressions for the coefficients appearing in Eq.~\eqref{eq:depth4-equation-intro} are given; see also §\ref{sec:nontriviality} about the non-triviality of the equation. For general~$q$, an insufficient supply of $\{2,q\}$-integral points on the thrice-punctured line makes it difficult to determine those coefficients. Taking $q = 3$, however, where $X(\bZ[1/6])$ contains for example the points~$-3$, $3$, $9$, we get a completely explicit equation:
	
	\begin{thm}{(= \Cref{thm:det-equation})}
		\label{thm:det-equation-intro}
		Let $S = \{2,3\}$ and $p \not\in S$. Any point $z$ in the refined Chabauty--Kim locus $X(\bZ_p)_{\{2,3\},4}^{(1,0)}$ satisfies the equation
		\begin{equation}
			\label{eq:det-equation-intro}
			\det \begin{pmatrix}
				\Li_4(z) & \log(z) \Li_3(z) & \log(z)^3 \Li_1(z) \\
				\Li_4(3) & \log(3) \Li_3(3) & \log(3)^3 \Li_1(3) \\
				\Li_4(9) & \log(9) \Li_3(9) & \log(9)^3 \Li_1(9) 
			\end{pmatrix} = 0.
		\end{equation}
	\end{thm}
	
	Using computations in Sage, we can verify for $S = \{2,3\}$ and many choices of~$p$ that all $p$-adic points satisfying both equations \eqref{eq:depth2-function-intro} and~\eqref{eq:det-equation-intro} are in fact $\{2,3\}$-integral points, thus proving \Cref{thm:main-kim-intro}.
	This provides evidence for Kim's Conjecture and supports the principle that, while a single Coleman function is usually insufficient to cut out precisely the set of $S$-integral points, two independent Coleman functions will often suffice.
	
	\subsection*{Structure of the paper}
	
	We start by recalling in §\ref{sec:background} the necessary background on refined Chabauty--Kim theory for the thrice-punctured line. We then derive in §\ref{sec:kim-functions} the Coleman functions which vanish on the depth~4 Chabauty--Kim loci in the case $S = \{2,q\}$, proving the precise version of \Cref{thm:main-function-intro} for general~$q$, as well as \Cref{thm:det-equation-intro} for $q = 3$. 
	We then turn to the computational aspects of this paper. In §\ref{sec:depth2-loci} we describe how to systematically compute the depth~2 loci $X(\bZ_p)_{\{2,q\},2}^{(1,0)}$ for arbitrary $p$ and $q$ and analyse the obtained data in §\ref{sec: sizes}. Finally, in §\ref{sec:verifying-kim} we present the computations which we use to verify instances of Kim's Conjecture in \Cref{thm:main-kim-intro}.

	\subsection*{Acknowledgements}
	I would like to thank Steffen Müller for his encouragement to write this paper and for helpful discussions. I also thank Elie Studnia and David Lilienfeldt for feedback on an earlier draft, and the referees for helpful comments. This work was supported by an NWO Vidi grant.

	\section{Background on refined Chabauty--Kim}
	\label{sec:background}
	
	We start by recalling what is known about the refined Chabauty--Kim method for the thrice-punctured line, referring to the existing literature for details. The original (unrefined) method is studied in \cites{kim:motivic}{DCW:explicitCK}{DCW:mixedtate1}{brown:IntegralPoints}[§8]{BDCKW}{CDC:polylog1}{CDC:polylog2}, the refined variant in \cite{BBKLMQSX, BKL:chabauty_kim_sc}. 
	
	Let $S$ be a finite set of primes and let $X = \bP^1 \smallsetminus \{0,1,\infty\}$ be the thrice-punctured line over the ring of $S$-integers $\bZ_S$. Let $p \not\in S$ be an auxiliary prime and let $U = \pi_1^{\et,\bQ_p}(X_{\overline{\bQ}},b)$ be the $\bQ_p$-prounipotent étale fundamental group of $X$ at the tangential base point $b = \vec{1}_0$, equipped with its natural action by the absolute Galois group~$G_{\bQ}$. For any finite-dimensional $G_{\bQ}$-equivariant quotient $U \twoheadrightarrow U'$, we have the following \emph{Chabauty--Kim diagram} \cite[§3]{kim:motivic} \cite[Introduction]{kim:albanese} 
	\begin{equation}
		\label{eq:CK-diagram}
		\begin{tikzcd}
			X(\bZ_S) \dar["j_S"] \rar[hook] & X(\bZ_p) \dar["j_p"] \\
			\rH^1_{f,S}(G_{\bQ}, U') \rar["\loc_p"] & \rH^1_f(G_p, U').
		\end{tikzcd}
	\end{equation}
	Here, $\rH^1_{f,S}(G_{\bQ}, U')$ denotes the Bloch--Kato Selmer scheme which parametrises $G_{\bQ}$-equivariant $U'$-torsors which are unramified outside $S \cup \{p\}$ and crystalline at~$p$; the local Selmer scheme $\rH^1_f(G_p, U')$ parametrises crystalline $G_p$-equivariant $U'$-torsors.\footnote{Kim \cite{kim:motivic, kim:albanese} considers the case that $U' = U_n$ is the maximal $n$-step nilpotent quotient but the construction of the Chabauty--Kim diagram works for arbitrary $G_{\bQ}$-equivariant quotients. Moreover, Kim writes $\rH^1_f(G_T, -)$ instead of $\rH^1_{f,S}(G_{\bQ}, -)$ where $T = S \cup \{p\}$ and $G_T$ is the Galois group of the maximal extension of~$\bQ$ unramified outside~$T$. This is an equivalent way of imposing the local conditions of being unramified outside $T$, cf.~\cite[§2.8]{BDCKW}.}
	Both the global and the local Selmer scheme are affine $\bQ_p$-schemes and the localisation map between them is algebraic. Strictly speaking, the vertical maps in the Chabauty--Kim diagram, both of which send a point $x$ of $X$ to the torsor of paths from $b$ to $x$, map into the $\bQ_p$-points of the schemes but this is customarily omitted from the notation. The \emph{Chabauty--Kim locus} for the fundamental group quotient $U'$ is defined as
	\[ X(\bZ_p)_{S,U'} \coloneqq j_p^{-1}(\loc_p(\rH^1_{f,S}(G_{\bQ}, U'))), \]
	in other words as the inverse image under~$j_p$ of the scheme-theoretic image of the Selmer scheme under the localisation map. If $U' = U_n$ is the maximal $n$-step nilpotent quotient of $U$, we denote the locus by $X(\bZ_p)_{S,n}$ and call it the Chabauty--Kim locus of \emph{depth}~$n$. The sets $X(\bZ_p)_{S,U'}$ all contain the set of $S$-integral points $X(\bZ_S)$ by construction and are cut out inside $X(\bZ_p)$ by Coleman functions, i.e., locally analytic functions given by iterated Coleman integrals. Whenever $f$ is an algebraic function on $\rH^1_f(G_p, U')$ such that $\loc_p^{\sharp}f = 0$, then $f \circ j_p$ is a Coleman function on $X(\bZ_p)$ which vanishes on $X(\bZ_p)_{S,U'}$. Determining the Chabauty--Kim loci in practice boils down to finding such functions.
	
	
	\subsection{The localisation map}
	\label{sec:localisation map}
	
	One fundamental group quotient which is particularly convenient to work with is the so-called \emph{polylogarithmic quotient} of depth~$n$, which we denote by $U_{\PL,n}$. We denote the associated Chabauty--Kim loci by $X(\bZ_p)_{S,\PL,n}$. The Coleman functions defining them involve only single polylogarithms $\Li_k$ ($1 \leq k \leq n$), as opposed to multiple polylogarithms $\Li_{k_1,\ldots,k_r}$ with $r \geq 2$.
	Thanks to prior work by Corwin and Dan-Cohen \cite{CDC:polylog1} we can write down the localisation map in the Chabauty--Kim diagram for $U_{\PL,n}$ quite explicitly. 
	The global and local Selmer scheme are both affine spaces over~$\bQ_p$. Taking $n = 4$, the global Selmer scheme is given by
	\[ \rH^1_{f,S}(G_{\bQ}, U_{\PL,4}) = \Spec \, \bQ_p[(x_{\ell})_{\ell \in S}, (y_{\ell})_{\ell \in S}, z_3] = \bA^S \times \bA^S \times \bA^1. \]
 	The functions $x_{\ell}$ and $y_{\ell}$ are canonical, whereas $z_3$ depends on certain choices. In \cite{CDC:polylog1}, these functions are denoted by $x_{\ell} = \Phi^{\tau_{\ell}}_{e_0}$, $y_{\ell} = \Phi^{\tau_{\ell}}_{e_1}$, $z = \Phi^{\sigma_3}_{e_1e_0e_0}$.
 	
 	The local Selmer scheme has a canonical set of coordinates
 	\[ \rH^1_f(G_p, U_{\PL,4}) = \Spec\, \bQ_p[\log,\Li_1,\Li_2,\Li_3,\Li_4] \]
 	coming from the non-abelian Bloch--Kato logarithm $\rH^1_f(G_p, U_{\PL,4}) \cong U_{\PL,4}^{\dR}$.	They are such that $\log(j_p(z)) = \log^p(z)$ is the $p$-adic logarithm for $z \in X(\bZ_p)$, and $\Li_n(j_p(z)) = \Li_n^p(z)$ is the $n$-th $p$-adic polylogarithm. (We often omit the superscript~$(-)^p$ from the notation.) These coordinates can be used to write down the localisation map
 	\begin{equation} 
 		\label{eq:localisation-map}
 		\loc_p\colon \rH^1_{f,S}(G_{\bQ}, U_{\PL,4}) \to \rH^1_f(G_p, U_{\PL,4}).
 	\end{equation}

	\begin{prop}
		\label{thm: localisation map}
		With respect to the coordinates above, the localisation map~\eqref{eq:localisation-map} for the polylogarithmic quotient in depth 4 is given as follows:
		\begin{align}
			\label{eq:loc-log}
			\loc_p^\sharp \log &= \sum_{\ell \in S} a_{\tau_{\ell}} x_{\ell},\\
			\label{eq:loc-Li1}
			\loc_p^\sharp \Li_1 &= \sum_{\ell \in S} a_{\tau_{\ell}} y_{\ell},\\
			\label{eq:loc-Li2}
			\loc_p^\sharp \Li_2 &= \sum_{\ell,q \in S} a_{\tau_{\ell} \tau_q} x_{\ell} y_q,\\
			\label{eq:loc-Li3}
			\loc_p^\sharp \Li_3 &= \sum_{\ell_1,\ell_2,q \in S} a_{\tau_{\ell_1} \tau_{\ell_2}\tau_q} x_{\ell_1} x_{\ell_2} y_q + a_{\sigma_3} z_3,\\
			\label{eq:loc-Li4}
			\loc_p^\sharp \Li_{4} &= \sum_{\ell_1,\ell_2,\ell_3,q \in S} a_{\tau_{\ell_1} \tau_{\ell_2} \tau_{\ell_3}\tau_{q}} x_{\ell_1} x_{\ell_2} x_{\ell_3} y_q + \sum_{\ell \in S} a_{\tau_{\ell} \sigma_3} x_{\ell} z_3.
		\end{align}
		Here, the $a_{u}$, subscripted by words in the symbols $\tau_{\ell}$ ($\ell \in S$) and $\sigma_3$, are certain $p$-adic constants.
	\end{prop}
	
	\begin{proof}
		This is \cite[Theorem~5.6]{BKL:chabauty_kim_sc} for $n=4$. The formulas are originally derived in {\cite[Corollary~3.11]{CDC:polylog1}} for a motivic variant of the Chabauty--Kim diagram, and transferred to the étale setting via a comparison theorem~\cite[Theorem~3.2]{BKL:chabauty_kim_sc}.
	\end{proof}
	
	\begin{rem}
		\label{rem:padic-coeffs}
		The $p$-adic constants $a_u$ appearing in the localisation map are hard to determine in practice. In §\ref{sec:explicit-equations} we discuss this issue and determine the values of the constants in the case $S = \{2,3\}$.
	\end{rem}
	
	\subsection{Refined Selmer schemes}
	\label{sec:refined}
	
	The refined Chabauty--Kim method by Betts and Dogra \cite{BD:refined} replaces the Selmer scheme $\rH^1_{f,S}(G_{\bQ}, U')$ by a certain closed subscheme $\Sel_{S,U'}^{\min}(X)$ called the \emph{refined Selmer scheme}. It is defined in such a way that it still fits into the commutative diagram~\eqref{eq:CK-diagram}. The refined Selmer scheme for the thrice-punctured line was first studied in \cite{BBKLMQSX} in depth~2 and later in \cite[§4]{BKL:chabauty_kim_sc} for general fundamental group quotients. If $2 \not\in S$, then both $X(\bZ_S)$ and the refined Selmer scheme are automatically empty as a consequence of $X(\bZ_2)$ being empty. \Cref{kim-conjecture-refined} is trivially satisfied in this case. So assume $2 \in S$ from now on. Then the refined Selmer scheme can be written as a union of $3^{\#S}$ closed subschemes as follows.
	
	For each $\ell \in S$ we have the mod-$\ell$ reduction map
	\[ \mathrm{red}_{\ell}\colon X(\bZ_S) \subseteq X(\bQ_{\ell}) \subseteq \bP^1(\bQ_{\ell}) = \bP^1(\bZ_{\ell}) \to \bP^1(\bF_{\ell}). \]
	Let $\Sigma = (\Sigma_{\ell})_{\ell \in S} \in \{0,1,\infty\}^S$ be a tuple consisting of a choice of a boundary point $\Sigma_{\ell} \in \{0,1,\infty\}$ for each $\ell \in S$. We call such a tuple a \emph{refinement condition}. (It corresponds roughly to the notion of \emph{reduction type} of \cite[§6.1]{betts:effective}.) Denote by $X(\bZ_S)_{\Sigma}$ the set of $S$-integral points~$z$ such that $\mathrm{red}_{\ell}(z) \in (X \cup \{\Sigma_{\ell}\})(\bF_{\ell})$ for all $\ell \in S$. Note that each $S$-integral point is contained in $X(\bZ_S)_{\Sigma}$ for some~$\Sigma$. (If the point is already $S'$-integral for a proper subset $S' \subsetneq S$, there are multiple possible choices of~$\Sigma$.) Associated to $\Sigma$ we have a partial refined Selmer scheme $\Sel^{\Sigma}_{S,U'}(X)$ fitting into a $\Sigma$-refined version of~\eqref{eq:CK-diagram}:
	\begin{equation}
		\label{eq:CK-diagram-refined}
		\begin{tikzcd}
			X(\bZ_S)_{\Sigma} \dar["j_S"] \rar[hook] & X(\bZ_p) \dar["j_p"] \\
			\Sel^{\Sigma}_{S,U'}(X) \rar["\loc_p"] & \rH^1_f(G_p, U').
		\end{tikzcd}
	\end{equation}
	This diagram is used to define the $\Sigma$-refined Chabauty--Kim locus
	\[ X(\bZ_p)_{S,U'}^{\Sigma} \coloneqq j_p^{-1}(\loc_p(\Sel^{\Sigma}_{S,U'}(X))). \]
	The total refined Selmer scheme $\Sel^{\min}_{S,U'}(X)$ is the union of closed subschemes
	\[ \Sel^{\min}_{S,U'}(X) = \bigcup_{\Sigma} \Sel^{\Sigma}_{S,U'}(X), \]
	with $\Sigma \in \{0,1,\infty\}^S$ running over all refinement conditions. Accordingly, the total refined Chabauty--Kim locus $X(\bZ_p)_{S,U'}^{\min}$ equals the union of the $\Sigma$-refined loci:
	\[  X(\bZ_p)_{S,U'}^{\min} = \bigcup_{\Sigma} X(\bZ_p)_{S,U'}^{\Sigma}. \]
	Diagram~\eqref{eq:CK-diagram-refined} implies that $X(\bZ_p)_{S,U'}^{\Sigma}$ contains the set $X(\bZ_S)_{\Sigma}$. It is natural to formulate Conjecture~\ref{kim-conjecture-refined} for each refinement condition $\Sigma$ separately:
	\begin{conjecture}[$\Sigma$-refined Kim's Conjecture]
		\label{kim-conjecture-sigma}
		$X(\bZ_p)_{S,n}^{\Sigma} = X(\bZ_S)_{\Sigma}$ for $n \gg 0$.
	\end{conjecture}
	Here, $X(\bZ_p)_{S,n}^{\Sigma}$ denotes the $\Sigma$-refined Chabauty--Kim locus for the depth~$n$ quotient~$U_n$. Clearly, if Conjecture~\ref{kim-conjecture-sigma} holds for each refinement condition $\Sigma$, then \Cref{kim-conjecture-refined} for the total refined locus $X(\bZ_p)_{S,n}^{\min}$ also holds. 
	
	\begin{rem}
		The converse is not true: for example, when $q > 3$ is a Fermat or Mersenne prime then \Cref{kim-conjecture-refined} holds for $X(\bZ_3)_{\{2,q\},2}^{\min}$ \cite[Cor.~3.15]{BBKLMQSX} whereas the $(1,0)$-refined variant fails, due to $2 \in X(\bZ_3)_{\{2,q\},2}^{(1,0)} \smallsetminus X(\bZ[\frac1{2q}])_{(1,0)}$ \cite[Rmk.~3.11]{BBKLMQSX}.
	\end{rem}
	
	Recall from §\ref{sec:localisation map} that the Selmer scheme $\rH^1_{f,S}(G_{\bQ}, U_{\PL,4}) \cong \bA^S \times \bA^S \times \bA^1$ for the polylogarithmic quotient of depth~$4$ carries canonical functions $x_{\ell}$ and $y_{\ell}$ for $\ell \in S$. The $\Sigma$-refined Selmer scheme can be described as a linear subspace in terms of these coordinates. The following is a special case of \cite[Prop.~5.11]{BKL:chabauty_kim_sc}.
	
	\begin{prop}
		\label{thm:refined-selmer-scheme-equations}
		Let $\Sigma = (\Sigma_{\ell})_{\ell \in S} \in \{0,1,\infty\}^S$ be a refinement condition. The $\Sigma$-refined Selmer scheme $\Sel^{\Sigma}_{S,\PL,4}(X)$ is the closed subscheme of $\rH^1_{f,S}(G_{\bQ}, U_{\PL,4})$ defined by the following equations for all~$\ell \in S$:
		\[ \begin{cases}
			y_{\ell} = 0, & \text{ if } \Sigma_{\ell} = 0,\\
			x_{\ell} = 0, & \text{ if } \Sigma_{\ell} = 1,\\
			x_{\ell} + y_{\ell} = 0, & \text{ if } \Sigma_{\ell} = \infty.
		\end{cases} \]
	\end{prop}

	\section{Refined Kim functions in depth~4}
	\label{sec:kim-functions}
	
	We now consider the case where $S = \{2,q\}$ for some odd prime~$q$. In this section we first reduce Conjecture~\ref{kim-conjecture-refined} in depth~$4$ to \Cref{kim-conjecture-sigma} for the polylogarithmic depth~4 quotient and for only two particular choices of the refinement condition~$\Sigma$. We then determine Coleman functions vanishing on the respective Chabauty--Kim sets, proving Theorems~\ref{thm:main-function-intro} and~\ref{thm:det-equation-intro} from the introduction.
	
	\subsection{Reducing Kim's conjecture}
	
	\begin{lemma}\leavevmode
		\label{thm:kim-conjecture-reduction}
			Assume that $X(\bZ_p)_{S,\PL,4}^{\Sigma}  = X(\bZ_S)_\Sigma$ holds for the two refinement conditions $\Sigma = (1,1)$ and $\Sigma = (1,0)$. Then Conjecture~\ref{kim-conjecture-sigma} holds in depth~4 for \emph{all} $\Sigma \in \{0,1,\infty\}^2$. In particular, Kim's Conjecture for the total refined locus (Conjecture~\ref{kim-conjecture-refined}) holds in depth~4.
	\end{lemma}
	
	\begin{proof}
		By \cite[Lemma~4.11]{BKL:chabauty_kim_sc}, the quotient map $U_4 \twoheadrightarrow U_{\PL,4}$ from the full depth~4 quotient to the polylogarithmic depth~4 quotient of the fundamental group induces an inclusion 
		\[ X(\bZ_p)_{S,4}^{\Sigma} \subseteq X(\bZ_p)_{S,\PL,4}^{\Sigma}. \]
		Thus, whenever $X(\bZ_p)_{S,\PL,4}^{\Sigma} = X(\bZ_S)_{\Sigma}$ holds, then also $X(\bZ_p)_{S,4}^{\Sigma} = X(\bZ_S)_{\Sigma}$. By \cite[Lemma~4.12]{BKL:chabauty_kim_sc}, the loci $X(\bZ_p)_{S,4}^{\Sigma}$ are functorial with respect to the $S_3$-action on $\bP^1 \smallsetminus \{0,1,\infty\}$. Specifically, for any automorphism $\sigma \in S_{\{0,1,\infty\}} \cong S_3$, given by one of the six Möbius transformations
		\[ z, \quad 1-z, \quad \frac1{z}, \quad \frac{z-1}{z}, \quad \frac{z}{z-1},\quad  \frac1{1-z}, \]
		and for any refinement condition $\Sigma \in \{0,1,\infty\}^2$, we have
		\[ \sigma(X(\bZ_p)_{S,4}^{\Sigma}) = X(\bZ_p)_{S,4}^{\sigma(\Sigma)}. \]
		In particular, if $X(\bZ_p)_{S,4}^{\Sigma} = X(\bZ_S)_{\Sigma}$ holds for some $\Sigma$, then it also holds for $\sigma(\Sigma)$. Any refinement condition is either of the form $\sigma((1,1))$ or $\sigma((1,0))$ for some $\sigma \in S_3$, so if Kim's Conjecture holds for $X(\bZ_p)_{S,4}^{(1,1)}$ and $X(\bZ_p)_{S,4}^{(1,0)}$, then it holds in fact for all refinement conditions, and thus for the total refined locus $X(\bZ_p)_{S,4}^{\min}$.
	\end{proof}

	\subsection{The $(1,1)$-locus}
	
	
	\begin{thm}
		\label{thm: 11 equations}
		The following equations hold on $X(\bZ_p)_{\{2,q\},\PL,4}^{(1,1)}$:
		\[ \log(z) = 0, \quad \Li_2(z) = 0, \quad \Li_4(z) = 0. \]
	\end{thm}

	\begin{proof}
		By \Cref{thm:refined-selmer-scheme-equations}, the refined Selmer scheme $\Sel^{(1,1)}_{S,\PL,4}(X)$ is the closed subscheme of  $\rH^1_{f,S}(G_{\bQ}, U_{\PL,4})$ 
		defined by $x_2 = x_q = 0$. 
		The restriction of the localisation map $\loc_p$ to this refined subscheme is given by setting $x_2$ and $x_q$ equal to zero in \Cref{thm: localisation map}.
%
		The functions $\log$, $\Li_2$, $\Li_4$ pull back to~$0$ on $\Sel^{(1,1)}_{S,\PL,4}(X)$, so their pullbacks along~$j_p$ vanish on $X(\bZ_p)_{\{2,q\},\PL,4}^{(1,1)}$.
	\end{proof}
	
	\begin{rem}
		\label{rem:11-locus}
		The set $X(\bZ_p)_{\{2,q\},\PL,4}^{(1,1)}$ for $S = \{2,q\}$ agrees with the set $X(\bZ_p)_{\{2\},\PL,4}^{(1)}$ for $S = \{2\}$; in particular, it is independent of the prime~$q$. This reflects the fact that $X(\bZ[\frac1{2q}])_{(1,1)} = \{-1\}$ for all~$q$. 
		The two equations $\log(z) = 0$ and $\Li_2(z) = 0$ were already derived for the depth~2 locus in \cite[Proposition~3.8]{BBKLMQSX}. We believe that these two functions suffice to cut out exactly the set $\{-1\}$ and we verified this computationally using the method described in Remark~3.6 of loc.\ cit.\ for all odd primes $p < 10^5$. Thus, Conjecture~\ref{kim-conjecture-sigma} for $\Sigma = (1,1)$ holds in depth~$2$ for those primes. 
		
		Using the localisation map in infinite depth \cite[Theorem~5.6]{BKL:chabauty_kim_sc} one sees easily that $X(\bZ_p)_{\{2,q\},\PL,n}^{(1,1)} = X(\bZ_p)_{\{2\},\PL,n}^{(1)}$ holds in fact for any depth~$n$. By Corollary~5.16 of loc.\ cit., the latter locus is exactly $\{-1\}$ when $n = \max(1,p-3)$. Thus, Conjecture~\ref{kim-conjecture-sigma} for $S = \{2,q\}$ and refinement condition $\Sigma = (1,1)$ holds for any choice of~$p$ in sufficiently high depth.
	\end{rem}
	
	\subsection{The $(1,0)$-locus}
	
	\begin{thm}
		\label{thm: 10 equations}
		The following two equations hold on the $(1,0)$-component of the refined Chabauty--Kim locus $X(\bZ_p)_{\{2,q\},\PL,4}^{(1,0)}$:
		\begin{align}
			\label{eq:depth2-function-thm}
			&a_{\tau_2} a_{\tau_q} \Li_2(z) - a_{\tau_q \tau_2} \log(z) \Li_1(z) = 0,\\
			\label{eq:depth4-function-thm}
			&a_{\sigma_3} a_{\tau_q}^3 a_{\tau_2} \Li_4(z) - a_{\tau_q}^2 a_{\tau_2} a_{\tau_q \sigma_3} \log(z) \Li_3(z) \\
			&\qquad - \bigl(a_{\sigma_3} a_{\tau_q\tau_q\tau_q\tau_2} - a_{\tau_q\sigma_3} a_{\tau_q \tau_q \tau_2}\bigr) \log(z)^3 \Li_1(z) = 0. \notag
		\end{align}
	\end{thm}

	\begin{proof}
		By \Cref{thm:refined-selmer-scheme-equations}, the refined Selmer scheme $\Sel^{(1,0)}_{S,\PL,4}(X)$ is the closed subscheme of $\rH^1_{f,S}(G_{\bQ}, U_{\PL,4})$ 
		defined by $x_2 = 0$ and $y_q = 0$. Denote the inclusion by~$i_\Sigma$. The restriction of the  localisation map $\loc_p$ to this refined subscheme is given by setting $x_2$ and $y_q$ equal to zero in \Cref{thm: localisation map}:
		\begin{align*}
			(\loc_p \circ i_\Sigma)^\sharp \log &= a_{\tau_q} x_q,\\
			(\loc_p \circ i_\Sigma)^\sharp \Li_1 &= a_{\tau_2} y_2,\\
			(\loc_p \circ i_\Sigma)^\sharp \Li_2 &= a_{\tau_q \tau_2} x_q y_2,\\
			(\loc_p \circ i_\Sigma)^\sharp \Li_3 &= a_{\tau_q\tau_q \tau_2} x_q^2 y_2 + a_{\sigma_3} z_3,\\
			(\loc_p \circ i_\Sigma)^\sharp \Li_4 &= a_{\tau_q \tau_q \tau_q \tau_2} x_q^3 y_2 + a_{\tau_q \sigma_3} x_q z_3.
		\end{align*}
		The linear combination $a_{\tau_2} a_{\tau_q} \Li_2 - a_{\tau_q \tau_2} \log \cdot \Li_1$ clearly pulls back to zero along $\loc_p \circ i_{\Sigma}$, which yields Equation~\eqref{eq:depth2-function-thm}. A slightly longer calculation yields the second equation: form a linear combination of $\Li_4$ and $\log \cdot \Li_3$ to eliminate the variable~$z_3$, resulting in a scalar multiple of $x_q^3 y_2$. Then form a linear combination with $\log^3 \cdot \Li_1$ to get a function pulling back to zero along $\loc_p \circ i_{\Sigma}$. 
	\end{proof}

	\begin{rem}
		\label{rem:depth2-equation}
		Equation~\eqref{eq:depth2-function-thm} alone defines the $(1,0)$-refined Chabauty--Kim locus in depth \emph{two}. It was first derived in \cite{BBKLMQSX} and can be rewritten in a more symmetric form as $a_{\tau_2 \tau_q} \Li_2(z) = a_{\tau_q \tau_2} \Li_2(1-z)$.
	\end{rem}
	
	\subsection{Nontrivial Kim functions}
	\label{sec:nontriviality}
	
	It is not known unconditionally whether Equation~\eqref{eq:depth4-function-thm} is nontrivial for every choice of auxiliary prime~$p$. We can however show that there is \emph{some} nontrivial equation of the same shape.
	
	\begin{thm}
		\label{thm:nontrivial-eq}
		There exists a nontrivial equation of the form
		\begin{equation}
			\label{eq:10-equation-simple}
			a \Li_4(z) + b \log(z) \Li_3(z) + c \log(z)^3 \Li_1(z) = 0
		\end{equation}
		with $a,b,c \in \bQ_p$ which holds on the Chabauty--Kim locus $X(\bZ_p)_{\{2,q\},\PL,4}^{(1,0)}$.
	\end{thm}
	
	Non-triviality means that the coefficients $a,b,c$ are not all zero. In this case, the left hand side of Eq.~\eqref{eq:10-equation-simple} is a nonzero Coleman function on $X(\bZ_p)$ by the Zariski-density of the $p$-adic unipotent Albanese map $j_p$ \cite[Theorem~1]{kim:albanese}. 
	
	\begin{proof}[Proof of \Cref{thm:nontrivial-eq}]
		In the proof of \Cref{thm: 10 equations}, the three monomials $\Li_4$, $\log\cdot \Li_3$, $\log^3 \cdot \Li_1$ all pull back under $\loc_p \circ i_{\Sigma}$ to linear combinations of the two monomials $x_q^3 y_2$, $x_q z_3$. This gives a $\bQ_p$-linear map from a $3$-dimensional space to a $2$-dimensional space, whose kernel must be nontrivial. 
	\end{proof}
	
	We can make \Cref{thm:nontrivial-eq} more concrete: we know that $a_{\tau_2} \neq 0$ and $a_{\tau_q} \neq 0$, so if we knew that $a_{\sigma_3} \neq 0$ as well, then the coefficient of $\Li_4$ in Equation~\eqref{eq:depth4-function-thm} would be nonzero and thus the equation would already be nontrivial. If $a_{\sigma_3} = 0$ on the other hand, then 
	\[ a_{\tau_q}^2 a_{\tau_2} \Li_3(z) - a_{\tau_q \tau_q \tau_2} \log(z)^2 \Li_1(z) = 0 \]
	would be a nontrivial equation which holds on $X(\bZ_p)_{\{2,q\},\PL,4}^{(1,0)}$. An equation of the form \eqref{eq:10-equation-simple} can then be obtained by multiplying by $\log(z)$.
	
	It is conjectured that $a_{\sigma_3} \neq 0$ for every choice of auxiliary prime~$p$. Indeed, the constant $a_{\sigma_3}$ equals the $p$-adic zeta value $\zeta(3)$ whose non-vanishing is implied by a $p$-adic period conjecture \cite[Conj.~2.25]{CDC:polylog1} \cite[Conj.~4]{yamashita:bounds}. The non-vanshing of $\zeta(3)$ is known when $p$ is a regular prime \cite[Rem.~2.20\,(i)]{furusho:p-adic}.
	
	\subsection{Explicit equations for $S = \{2,3\}$}
	\label{sec:explicit-equations}
	
	Determining the $p$-adic constants $a_u$ which appear in the equations \eqref{eq:depth2-function-thm} and \eqref{eq:depth4-function-thm} is difficult in general. It is complicated by the fact that their values depend on a choice of free generators $\{\tau_{\ell} : \ell \in S\} \cup \{\sigma_3,\sigma_5,\ldots\}$ for the Lie algebra of the unipotent mixed Tate Galois group $U_{\bQ,S}^{\MT}$ (see \cite[§4.1]{CDC:polylog1} for details). Those constants $a_u$ with a single letter as subscript are however canonical: $a_{\tau_\ell} = \log(\ell)$ for $\ell \in S$ are $p$-adic logarithms, and $a_{\sigma_3} = \zeta(3)$ is a $p$-adic zeta value. The values $a_{\tau_q \tau_2}$ which appear in Eq.~\eqref{eq:depth2-function-thm} are examples of ``Dan-Cohen--Wewers coefficients''. They are also canonical and we recall in §\ref{sec:DCW-coeffs} how to compute them. 
	The constants in Eq.~\eqref{eq:depth4-function-thm} are not known for a general prime~$q$. They have however been computed in the case $S = \{2,3\}$, exploiting the fact that $-3$,~$3$ and~$9$ are known $\bZ[1/6]$-integral points of the thrice-punctured line.
	
	\begin{prop}
		\label{thm: periods}
		For a suitable choice of $\sigma_3$ we have:
		\begin{align*}
			a_{\tau_3 \tau_2} &= -\Li_2(3),\\
			a_{\tau_3 \tau_3 \tau_2} &= -\Li_3(3),\\
			a_{\tau_3 \tau_3 \tau_3 \tau_2} &= -\Li_4(3),\\
			a_{\tau_3 \sigma_3} &= \frac{18}{13} \Li_4(3) - \frac3{52} \Li_4(9).
		\end{align*}
	\end{prop}
	
	\begin{proof}
		The value of $a_{\tau_3 \tau_2}$ is derived in \cite[§3.5]{BBKLMQSX}. The coefficients $a_{\tau_3 \tau_3 \tau_3 \tau_2}$ and $a_{\tau_3 \sigma_3}$ are determined in \cite[§4.3.3]{CDC:polylog1}, and a proof for $a_{\tau_3 \tau_3 \tau_2}$ (using the same choice of~$\sigma_3$) can be found in \cite[§5.11]{DCJ:M05}. Note that the authors use different notation and conventions. For example $a_{\tau_3 \tau_3 \tau_3 \tau_2}$ is written $f_{\tau_2 \tau \tau \tau}$ in \cite{CDC:polylog1}.
		A different expression for $a_{\tau_3 \tau_3 \tau_3 \tau_2}$ is given in \cite[§5.23]{DCJ:M05} (denoted $f_{\tau vvv}$ there) but it appears to be incorrect based on numerical evaluation.
	\end{proof}
	
	With this the equations for $X(\bZ_p)_{\{2,3\},\PL,4}^{(1,0)}$ from \Cref{thm: 10 equations} are completely determined. It is in principle possible to compute the constants $a_u$ for more general sets~$S$; see \cite{DC:mixedtate2} for an algorithm which achieves this under various conjectures. 
	But if we have enough $S$-integral points available, which only is the case for $S = \{2,3\}$, there is a different way to obtain an equation of the form \eqref{eq:10-equation-simple} which circumvents the problem of determining the $a_u$, and does not require any non-canonical choices. It uses an argument similar to \cite[Cor.~9.1]{brown:IntegralPoints}.
	
	\begin{thm}
		\label{thm:det-equation}
		Any $z$ in the Chabauty--Kim locus $X(\bZ_p)_{\{2,3\},\PL,4}^{(1,0)}$ satisfies
		\begin{equation}
			\label{eq:det-equation}
			\det \begin{pmatrix}
				\Li_4(z) & \log(z) \Li_3(z) & \log(z)^3 \Li_1(z) \\
				\Li_4(3) & \log(3) \Li_3(3) & \log(3)^3 \Li_1(3) \\
				\Li_4(9) & \log(9) \Li_3(9) & \log(9)^3 \Li_1(9) 
			\end{pmatrix} = 0.
		\end{equation}
	\end{thm}
	
	\begin{proof}
		By \Cref{thm:nontrivial-eq}, some nontrivial equation of the form~\eqref{eq:10-equation-simple}
		holds on the Chabauty--Kim locus. This means that the vectors 
		\[ (\Li_4(x), \log(x)\Li_3(x), \log(x)^3 \Li_1(x))\] 
		for $x \in X(\bZ_p)_{\{2,3\},\PL,4}^{(1,0)}$ lie in some $2$-dimensional linear subspace of $\bQ_p^3$. In particular, any three vectors of this form are linearly dependent. Taking $x = 3$ and $x = 9$ gives two such vectors since these points belong to $X(\bZ[1/6])_{(1,0)}$. 
	\end{proof}
	
	\begin{rem}
		\label{rem:brown-comparison}
		In \cite[Cor.~9.1]{brown:IntegralPoints}, determinant equations like Eq.~\eqref{eq:det-equation} are derived for \emph{unrefined} Chabauty--Kim loci $X(\bZ_p)_{S,\PL,n}$ given a sufficient supply of $S$-integral points. When $S$ has size two, depth~$n=4$ is not sufficient for the unrefined Chabauty--Kim locus to be finite. One needs at least $n = 6$ but then one would need more than $w = 252$ $S$-integral points to get a determinant equation. Going to even higher depth reduces this to $w = 64$ but this is still too large since $X(\bZ[1/2q])$ contains only~21 points for $q = 3$ and even fewer for $q > 3$. 
	\end{rem}

	\section{Computing Chabauty--Kim loci in depth~2}
	\label{sec:depth2-loci}
	
	Let $S = \{2,q\}$ for any odd prime~$q$ and let $p \not\in S$ be a choice of auxiliary prime. In this section we investigate the depth~2 Chabauty--Kim loci 
	\[ X(\bZ_p)_{\{2,q\},2}^{(1,0)} \] 
	for various combinations of $p$ and $q$. 
	By \cite[Proposition~3.9]{BBKLMQSX}, the locus is cut out in $X(\bZ_p)$ by the single equation
	\begin{equation}
		\label{eq:depth2-function}
		a_{\tau_2} a_{\tau_q} \Li_2(z) - a_{\tau_q \tau_2} \log(z) \Li_1(z) = 0.
	\end{equation}
	This is the first of the two functions found in \Cref{thm: 10 equations} above. We have an inclusion $X(\bZ_p)_{\{2,q\},2}^{(1,1)} \subseteq X(\bZ_p)_{\{2,q\},2}^{(1,0)}$ as the equations for the $(1,1)$-locus, $\log(z) = \Li_2(z) = 0$, imply Equation~\eqref{eq:depth2-function}. As a consequence, once we know the locus $X(\bZ_p)_{\{2,q\},2}^{(1,0)}$, we actually know the total refined Chabauty--Kim locus $X(\bZ_p)_{\{2,q\},2}^{\min}$ by taking $S_3$-orbits (cf. \cite[Theorem~B]{BBKLMQSX}).
	
	The problem of finding the solutions in $X(\bZ_p)$ to Equation~\eqref{eq:depth2-function} is resolved by the code accompanying this paper \cite{sage_code}.
	The task can be broken down into the following steps:
	\begin{enumerate}[label=\arabic*.]
		\item Determine the constant $a_{\tau_q \tau_2}$ appearing in Eq.~\eqref{eq:depth2-function}. 
		\item For each residue disc in $X(\bZ_p)$, compute a power series representing the Coleman function on the left hand side of Eq.~\eqref{eq:depth2-function} on that disc.
		\item For each residue disc, find the $p$-adic roots of the power series.
	\end{enumerate}
	
	Concerning step~1, note that only $a_{\tau_q \tau_2}$ needs to be determined; the constants $a_{\tau_2} = \log(2)$ and $a_{\tau_q} = \log(q)$ are simply given by $p$-adic logarithms.
	
	\subsection{An example}
	\label{sec: depth2-example}
	Before describing the steps of the algorithm in more detail, we demonstrate the function 
	\begin{center}
		\texttt{CK\_depth\_2\_locus(p, q, N, a\_q2)}
	\end{center}
	of the accompanying Sage code. Assume we want to compute the locus $X(\bZ_5)_{\{2,3\},2}^{(1,0)}$, i.e., we are looking at the thrice-punctured line over~$\bZ[1/6]$ (so $q = 3$) and choose the auxiliary prime $p = 5$. The argument~$N$ specifices the $p$-adic precision for the coefficients of power series, and the argument \texttt{a\_q2} is the constant $a_{\tau_q \tau_2}$, which in the case $q = 3$ is given by $a_{\tau_3 \tau_2} =-\Li_2(3)$ by \Cref{thm: periods}.
	The code
	\usemintedstyle{borland}
	\begin{minted}[frame=lines,baselinestretch=1.2,bgcolor=gray!8,
fontsize=\footnotesize]{python}
p = 5; q = 3
a = -Qp(p)(3).polylog(2)
CK_depth_2_locus(p,q,10,a)
	\end{minted}
	outputs the following list of six $5$-adic numbers:
	\begin{minted}[frame=lines,baselinestretch=1.2,bgcolor=gray!8,fontsize=\footnotesize]{python}
[2 + O(5^9),
 2 + 4*5 + 4*5^2 + 4*5^3 + 4*5^4 + 4*5^5 + 4*5^6 + 4*5^7 + 4*5^8 + O(5^9),
 3 + O(5^6),
 3 + 5^2 + 2*5^3 + 5^4 + 3*5^5 + O(5^6),
 4 + 4*5 + 4*5^2 + 4*5^3 + 4*5^4 + 4*5^5 + 4*5^6 + 4*5^7 + 4*5^8 + O(5^9),
 4 + 5 + O(5^9)]
	\end{minted}
	These form the refined Chabauty--Kim locus $X(\bZ_5)_{\{2,3\},2}^{(1,0)}$. This locus must contain the $\{2,3\}$-integral points $\{-3,-1,3,9\}$, as these are the points in $X(\bZ[1/6])$ whose mod-2 reduction lies in $X \cup \{1\}$ and whose mod-3 reduction lies in $X \cup \{0\}$. Indeed, we recognise those among the numbers in the list. We moreover observe that $z = 2$ satisfies Eq.~\eqref{eq:depth2-function} as a consequence of $\Li_2(2) = 0$ and $\Li_1(2) = -\log(1-2) = 0$, so it belongs to $X(\bZ_5)_{\{2,3\},2}^{(1,0)}$ as well. It is a $\{2,3\}$-integral (even $\{2\}$-integral) point of~$X$ but we would not a priori expect it to be part of the $(1,0)$-refined locus because it reduces to $0$, not~$1$, modulo~$2$. Indeed we will see later that it does not survive in depth~4. Finally, there is the exceptional point
	\begin{equation}
		\label{eq:extra-point}
		\mathtt{z_0 = 3 + 5^2 + 2*5^3 + 5^4 + 3*5^5 + O(5^6)}
	\end{equation}
	which is not a $\{2,3\}$-integral point of~$X$ and in fact seems to be transcendental over the rationals. As we noted at the end of \cite{BBKLMQSX}, this point is responsible for Kim's Conjecture not holding in depth~2. Again, we will see later that this exceptional solution can be ruled out by going to depth~4.

	We now describe the steps for computing $X(\bZ_p)^{(1,0)}_{\{2,q\},2}$ in general.
	
	\subsection{Computing DCW coefficients}
	\label{sec:DCW-coeffs}
	
	The first step consists in computing the $p$-adic constant $a_{\tau_q \tau_2}$, an example of what is called a ``Dan-Cohen--Wewers coefficient'' (and denoted by $a_{q,2}$) in \cite{BBKLMQSX}. In some cases, by §3.5 and Lemma~2.2 of loc.\ cit., we have a simple expression for $a_{\tau_q \tau_2}$:
	\begin{enumerate}
		\item For $q = 3$ we have $a_{\tau_3 \tau_2} = -\Li_2(3) = -\frac12 \Li_2(-3) = -\frac16 \Li_2(9)$.
		\item If $q = 2^n + 1$ is a Fermat prime, we have $a_{\tau_q \tau_2} = -\frac1{n} \Li_2(q)$.
		\item If $q = 2^n-1$ is a Mersenne prime, we have $a_{\tau_q \tau_2} = -\frac1{n} \Li_2(-q)$.
	\end{enumerate}
	For general primes $q$ we only have an algorithm for expressing $a_{\tau_q \tau_2}$ as a $\bQ$-linear combination of dilogarithms of rational numbers. In general, the DCW coefficients $a_{\tau_{\ell} \tau_q}$ for $\ell,q \in S$ appear in the localisation map of the Chabauty--Kim diagram~\eqref{eq:CK-diagram}. By \Cref{thm: localisation map}, they are the coefficients in the bilinear polynomial $\loc_p^{\sharp}\Li_2 = \sum_{\ell,q\in S} a_{\tau_{\ell} \tau_q} x_{\ell} y_q$.
	An algorithm for computing $a_{\tau_{\ell} \tau_q}$ for any pair of primes $\ell,q \neq p$ is described in \cite[§2.3]{BBKLMQSX}, slightly generalising the original algorithm from \cite[§11]{DCW:explicitCK}. In short (and simplifying a bit), one considers the $\bQ$-vector space
	\[ E \coloneqq \bQ \otimes \bQ^\times. \]
	Write $[x] \coloneqq 1 \otimes x \in E$ for $x \in \bQ^\times$.
	By Tate's vanishing of the rational Milnor $K$-group $K_2(\bQ) \otimes {\bQ}$ \cite[Theorem~11.6]{milnor}, every element of $E \otimes E$ can be written as a $\bQ$-linear combination of Steinberg elements $[t] \otimes [1-t]$ where $t \in \bQ \smallsetminus \{0,1\}$. Given such a ``Steinberg decomposition'' in $E \otimes E$
	\begin{equation}
		\label{eq:steinberg decomposition}
		[\ell] \otimes [q] = \sum_i c_i\, [t_i] \otimes [1-t_i]
	\end{equation}
	for a pair of primes $(\ell,q)$, a formula for the DCW coefficient $a_{\tau_{\ell} \tau_q}$ is given by
	\begin{equation}
		\label{eq:DCW formula}
		a_{\tau_{\ell} \tau_q} = -\sum_i c_i \Li_2(t_i).
	\end{equation}
	The algorithm described in \cite{BBKLMQSX} is slightly more involved in that it computes Steinberg decompositions in the wedge square $E \wedge E$ rather than the tensor square (for efficiency) and takes care to avoid Steinberg elements $[t] \wedge [1-t]$ where $t$ or $1-t$ contains factors of~$p$ (to avoid choosing a branch of the $p$-adic logarithm). In loc.~cit., the right hand side of Eq.~\eqref{eq:steinberg decomposition} also contains additional terms of the form $[x] \otimes [y] + [y] \otimes [x]$, but those are expressable in terms of Steinberg elements as well, so they can be subsumed under the sum in Eq.~\eqref{eq:steinberg decomposition}. The algorithm described there is implemented in Sage \cite{KLS:dcwcoefficients}. It can be used to compute a Steinberg decomposition of $[2] \wedge [q]$ in $E \wedge E$, which can then be passed to the function
		\texttt{depth2\_constant}
	from \cite{sage_code} to compute the $p$-adic constant $a_{\tau_q \tau_2}$ with precision~$N$. 
	For example, taking $q = 19$ and $p = 7$, the code
\begin{minted}[frame=lines,baselinestretch=1.2,bgcolor=gray!8,fontsize=\footnotesize]{python}
p = 7; q = 19
_,dec = steinberg_decompositions(bound=20, p=p)
depth2_constant(p, q, 10, dec[2,q])
\end{minted}
	computes $a_{\tau_{19} \tau_2} \in \bQ_7$ as
	\[ \mathtt{ a_{\tau_{19} \tau_2} = 7^2 + 2*7^3 + 6*7^4 + 3*7^5 + 2*7^6 + 6*7^7 + 5*7^8 + O(7^{10})}. \]

	\subsection{Computing power series}
	
	The second step in the computation of the Chabauty--Kim locus $X(\bZ_p)_{\{2,q\},2}^{(1,0)}$ consists in computing on each residue disc a $p$-adic approximation of the power series representing the defining Coleman function~\eqref{eq:depth2-function}.
	We already discussed the computation of the coefficients, so the problem is reduced to finding power series for the polylogarithmic functions $\log(z)$, $\Li_1(z)$, and $\Li_2(z)$. Later we will also need the power series for $\Li_3(z)$ and $\Li_4(z)$, so we discuss here the general problem of computing power series for $\Li_n(z)$ for arbitrary~$n$. In order to achieve this, we adapt the work by Besser and de~Jeu \cite{besser:lipservice} which contains an algorithm for computing $\Li_n(z)$ for any $z \in \bC_p \smallsetminus \{1\}$. We can exploit the fact that we work in $\bZ_p$ rather than $\bC_p$ to improve the convergence of the power series. 
	
	For each $(p-1)$-st root of unity $\zeta$ in $\bZ_p$ write $U_{\zeta}$ for its residue disc, consisting of all elements of $\bZ_p$ reducing to $\zeta$ modulo~$p$. Those residue discs $U_{\zeta}$ with $\zeta \neq 1$ cover $X(\bZ_p)$. Any element $z \in U_{\zeta}$ can be written as $z = \zeta + pt$ with $t \in \bZ_p$ and we wish to compute $p$-adic approximations of the power series in the parameter~$t$ for $\log(\zeta + pt)$ and for $\Li_m(\zeta + pt)$ for various $m$. We denote the coefficients of the latter series by $(a_{m,k})_{k \geq 0}$, so that we have
	\begin{equation}
		\label{eq:polylog-series}
		\Li_m(\zeta + pt) = \sum_{k = 0}^\infty a_{m,k} t^k.
	\end{equation}
	
	\begin{defn}
		\label{def:approximation}
		Let $f(t) = \sum_{k=0}^\infty a_k t^k \in \bQ_p[\![t]\!]$ be a convergent power series on the unit disc (i.e., $|a_k|_p \to 0$). A \emph{$p$-adic approximation of order~$N$} of $f(t)$ consists of a nonnegative integer $k_0$ along with a polynomial $\tilde f(t) = \sum_{k = 0}^{k_0-1} \tilde a_{k} t^k \in \bQ_p[t]$ of degree $< k_0$ such that $v_p(a_k - \tilde a_k) \geq N$ for $k < k_0$ and $v_p(a_k) \geq N$ for all $k \geq k_0$.
	\end{defn}
	
	For the $p$-adic logarithm, computing the power series is straightforward:
	\begin{lemma}
		\label{thm:log-series}
		Given a prime~$p$ and a $(p-1)$-st root of unity $\zeta \neq 1$ in $\bZ_p$, the series expansion of $\log(\zeta + pt)$ on $U_{\zeta}$ is given by
		\begin{equation}
			\label{eq:log-series}
			\log(\zeta + pt) = - \sum_{k=1}^\infty \frac{(-p)^k}{k \zeta^k} t^k.
		\end{equation}
		If $N \in \bZ_{\geq 0}$ is the desired precision, let $k_0$ be the smallest integer $\geq 1$ satisfying $k_0 - \log_p(k_0) \geq N$. Then for all $k \geq k_0$, the coefficient of $t^k$ in~\eqref{eq:log-series} has valuation $\geq N$. If $\tilde \zeta \in \bZ_p$ is an approximation of $\zeta$ of order~$N$ (i.e., $v_p(\zeta - \tilde \zeta) \geq N$), then
		$-\sum_{k=1}^{k_0-1}\frac{(-p)^k}{k \tilde \zeta^k} t^k$
		is an approximation of~\eqref{eq:log-series} of order~$N$.
	\end{lemma}
	
	\begin{proof}
		Since $\log(\zeta) = 0$ for roots of unity, we have 
		\[ \log(\zeta + pt) = \log(1 + pt/\zeta) = - \sum_{k=1}^\infty \frac{(-1)^k}{k} (pt/\zeta)^k = - \sum_{k=1}^\infty \frac{(-p)^k}{k \zeta^k} t^k \]
		as claimed. Let $c_k = - \frac{(-p)^k}{k \zeta^k}$ be the coefficient of $t^k$, then since $v_p(k) \leq \lfloor \log_p(k) \rfloor$, the valuation of $c_k$ satisfies
		$v_p(c_k) = k - v_p(k) \geq k - \log_p(k)$.
		The real-valued function $x \mapsto x - \log_p(x)$ is increasing to the right of its minimum at $x = 1/\log(p)$, in particular it is increasing in the range $x \geq 1$. (The assumptions on $\zeta$ imply $p \geq 3$.) Therefore, if the inequality $k_0 - \log_p(k_0) \geq N$ holds for some $k_0 \geq 1$, it also holds for every $k \geq k_0$. This implies that truncating the power series at the $k_0$-th term gives an approximation of order~$N$. Finally, if $\tilde \zeta$ approximates $\zeta$ to order~$N$, then 
		$1/\tilde \zeta$ approximates $1/\zeta$ to order~$N$ as well. Since $v_p((-p)^k/k) \geq 1$ for all $k \geq 1$, one can replace $\zeta$ with $\tilde \zeta$ in \eqref{eq:log-series}.
	\end{proof}
	
	For the polylogarithms $\Li_m$, consider the following computation problem:
	\begin{problem}
		\label{problem:polylog-series}
		Given a prime~$p$, a $(p-1)$-st root of unity $\zeta \neq 1$ in $\bZ_p$, a nonnegative integer~$n \in \bZ_{\geq 0}$, and a precision $N \in \bZ_{\geq 0}$, compute $p$-adic approximations of order~$N$ of the power series~\eqref{eq:polylog-series} of $\Li_m(\zeta + pt)$ for $m=0,\ldots,n$.
	\end{problem}
	
	This problem can be split into two subproblems. In order to compute the power series for $\Li_m(\zeta + pt)$ for $m =0,\ldots,n$ one first needs to compute the values $\Li_m(\zeta)$ at the root of unity~$\zeta$, in other words, the constant coefficients of the power series. The algorithm described in \cite{besser:lipservice} proceeds as follows. One considers the modified polylogarithm $\Li_m^{(p)}(z) \coloneqq \Li_m(z) - \frac1{p^m} \Li_m(z^p)$. It admits a power series expansion around~$\infty$ of the form
	$\Li_m^{(p)}(z) = g_m(1/(1- z))$
	with $g_m(v) \in \bQ[\![v]\!]$, converging for $|v|_p < p^{1/(p-1)}$. (Note that $g_m(v)$ depends on~$p$, even though this is not apparent from the notation.) For a $(p-1)$-st root of unity~$\zeta \neq 1$ one has 
	\[ \Li_m(\zeta) = \frac{p^m}{p^m-1} \Li_m^{(p)}(\zeta) = \frac{p^m}{p^m-1} g_m(1/(1-\zeta)). \]
	So we turn attention to the following problem:
	
	\begin{problem}
		\label{problem:g-series}
		Given a prime~$p$, an integer~$n \in \bZ_{\geq 0}$, and a precision $M \in \bZ_{\geq 0}$, compute $p$-adic approximations of order~$M$ of $g_m(v) \in \bQ[\![v]\!]$ for $m = 0,\ldots,n$.
	\end{problem}
	
	The series $g_m(v)$ can be computed recursively. For $m = 0$, we have
	\begin{equation}
		\label{eq:g0}
		g_0(v) = -(1-v) + (1-v)^p \sum_{i = 0}^\infty (p f(v))^i,
	\end{equation}
	where $f(v) \in v \bZ[v]$ is the polynomial of degree~$p-1$ defined by $(1-v)^p - (-v)^p = 1-pf(v)$.\footnote{Note that the formulas for $g_0(v)$ and $f(v)$ given in the proof of \cite[Prop.~6.1]{besser:lipservice} contain sign errors.} For $m \geq 1$, the power series $g_m(v)$ can be computed from $g_{m-1}(v)$ via
	\begin{equation}
		\label{eq:gm}
		g_m'(v) = -\frac{g_{m-1}(v)}{v}(1 + v + v^2 + \ldots)
	\end{equation}
	and $g_m(0) = 0$. We know estimates for the valuations of the coefficients of $g_m(v)$:
	
	\begin{lemma}{\cite[Prop.~6.1]{besser:lipservice}}
		\label{thm:g-estimates}
		For $m \geq 1$, let $g_m(v) = b_{m,1} v + b_{m,2} v^2 + \ldots$ be the power series expansion of $g_m(v) \in \bQ[\![v]\!]$. The valuations of the coefficients satisfy
		\[ v_p(b_{m,k}) \geq \max\left(0, \frac{k}{p-1} - \log_p(k) - c(m,k)\right) \]
		for some explicit constant $c(m,p) > 0$.
	\end{lemma}
	
	Here is how we solve \Cref{problem:g-series}. Given~$p$, $n$, and~$M$, we first determine $k_0 \in \bZ_{\geq 2}$ such that $\frac{k}{p-1} - \log_p(k) - c(m,p) \geq M$ for all $k \geq k_0$ and $m = 0,\ldots,n$. This is straightforward, exploiting the monotonicity properties of the real-valued function $x \mapsto \frac{x}{p-1} - \log_p(x) - c(m,p)$. Then, by \Cref{thm:g-estimates}, in order to approximate $g_m(v)$ to $p$-adic order~$M$ it suffices to compute the first $k_0$ terms. We start by computing those terms for $g_0(v)$ using Eq.~\eqref{eq:g0}. Then for $m \geq 1$, we use the recursive formula~\eqref{eq:gm}, which in terms of the coefficients $b_{m,k}$ reads
	\begin{equation}
		\label{eq:g-coeff-recursion}
		b_{m,k} = -\frac1{k}(b_{m-1,1} + \ldots + b_{m-1,k}) \quad \text{for $k \geq 1$},
	\end{equation}
	as well as $b_{m,0} = 0$. Even though all coefficients are rational numbers, it is more efficient to only store $p$-adic approximations. We can estimate how many $p$-adic digits we need: going from $m-1$ to $m$, the recursive formula~\eqref{eq:g-coeff-recursion} for $b_{m,k}$ involves a division by~$k$ which decreases the precision by~$v_p(k)$. But we are only computing coefficients with $k < k_0$, so the loss of precision can be bounded by $\delta \coloneqq \lfloor \log_p(k_0 - 1) \rfloor$. Knowing $g_0(v)$ to $p$-adic precision~$M + n\delta$ is therefore enough to subsequently compute all of $g_1(v),\ldots,g_n(v)$ to precision at least~$M$.
		
	We now return to \Cref{problem:polylog-series} asking for $p$-adic approximations of $\Li_m(\zeta + pt)$. They are computed by iterated integration, using the values $\Li_n(\zeta)$ as integration constants:
	
	\begin{lemma}
		\label{thm:polylog-series-formulas}
		The coefficients of $\Li_m(\zeta + pt) = \sum_{k = 0}^\infty a_{m,k} t^k$ are recursively given as follows. For $m = 0$ we have
		\begin{equation}
			\label{eq:Li0-series}
			\Li_0(\zeta + pt) = \frac{\zeta}{1-\zeta} + \sum_{k=1}^\infty \frac{p^k}{(1-\zeta)^{k+1}} t^k. 
		\end{equation}
		For $m \geq 1$ we have $a_{m,0} = \Li_m(\zeta) = \frac{p^m}{p^m-1} g_m(1/(1 - \zeta))$ 
		and
		\begin{equation}
			\label{eq:Lim-series}
			a_{m,k} = -\frac1{k} \sum_{j=0}^{k-1} \left(-\frac{p}{\zeta}\right)^{k-j} a_{m-1,j} \quad \text{ for $k \geq 1$}.
		\end{equation}
	\end{lemma}
	
	\begin{proof}
		We have $\Li_0(z) = \frac{z}{1-z}$ and $\rd\!\Li_m(z) = \Li_{m-1} \frac{\rd z}{z}$ for $m \geq 1$ by definition of the polylogarithms as iterated Coleman integrals. An easy calculation shows that this translates into the given formulas for the power series coefficients.
	\end{proof}
	
	\begin{lemma}
		\label{thm:polylog-series-estimates}
		The valuations of the coefficients $a_{m,k}$ satisfy $v_p(a_{m,0}) \geq m$ and $v_p(a_{m,k}) \geq k - m \lfloor \log_p(k) \rfloor$ for $k \geq 1$. 
	\end{lemma}
	
	\begin{proof}
		For $k = 0$, it follows from $a_{m,0} = \Li_m(\zeta) = \frac{p^m}{p^m-1} g_m(1/(1 - \zeta))$ and the fact that the coefficients of $g_m(v)$ have valuation $\geq 0$ that $a_{m,0}$ has valuation~$\geq m$. For $m = 0$, we see from Eq.~\eqref{eq:Li0-series} that the coefficient of $t^k$ has valuation exactly~$k$. Let $m \geq 1$ and assume the hypothesis for $m-1$. In the recursive formula~\eqref{eq:Lim-series} for computing the $a_{m,k}$ from the $a_{m-1,j}$ with $j < k$, the summands with $j \geq 1$ satisfy
		\begin{align*} 
			v_p\left(-\frac1{k} \left(-\frac{p}{\zeta}\right)^{k-j} a_{m-1,j}\right) &= k-j - v_p(k) + v_p(a_{m-1,j})\\
			&\geq k-j - v_p(k) + j - (m-1) \lfloor \log_p(j) \rfloor \\
			&\geq k - m \lfloor \log_p(k) \rfloor.
		\end{align*}
		The summand for $j = 0$ also satisfies this since $v_p(a_{m-1,0}) \geq m-1$.
		Hence, from Eq.~\eqref{eq:Lim-series} we find $v_p(a_{m,k}) \geq k - m \lfloor \log_p(k) \rfloor$ as claimed. 
%
	\end{proof}
	
	In order to compute $p$-adic approximations of order~$N$ to the power series of $\Li_m(\zeta + pt)$ for $m = 0,\ldots,n$, we first determine $k_0 \geq 1$ such that $k - n \lfloor \log_p(k)\rfloor \geq N$ for all $k \geq k_0$. 
	Then, by \Cref{thm:polylog-series-estimates}, it suffices to find order~$N$ approximations of the coefficients $a_{m,k}$ with $k < k_0$. In the recursive formula~\eqref{eq:Lim-series} for $a_{m,k}$, the division by $k$ causes a loss of precision by $v_p(k) \leq \lfloor \log_p(k_0 - 1) \rfloor$. On the other hand, each $a_{m-1,j}$ is multiplied by at least one factor of~$p$. Thus, in order to compute $a_{m,k}$ to precision~$N$ it suffices to compute the $a_{m-1,j}$ to precision~$N + \delta - 1$, where 
	\[ \delta \coloneqq \lfloor \log_p(k_0 - 1) \rfloor.\]
	This tells us two things: firstly, setting
	\[ M \coloneqq \max(N, N + n(\delta-1)), \]
	it suffices to compute the power series $g_m(v)$ for $m = 0,\ldots,n$ to precision~$M$. This is achieved by the algorithm for \Cref{problem:g-series} above. Secondly, we get an estimate on how many $p$-digits we need to store for each coefficient.
	
	\begin{lemma}
		\label{thm:digits-estimate}
		In order to compute $a_{m,k}$ for $k < k_0$ and $m = 0,\ldots,n$ with precision~$N$, no more than $M$ $p$-adic digits are needed for each coefficient.
	\end{lemma}
	
	\begin{proof}
		If $\delta \leq 1$, no precision is lost and no negative valuations occur when going from $m-1$ to $m$, so knowing the first $N$ ($=M$) $p$-adic digits of the coefficients $a_{0,k} \in \bZ_p$ is enough to compute the first $N$ digits of $a_{m,k} \in \bZ_p$ for all $m = 1,\ldots,n$. If $\delta \geq 2$ on the other hand, we compute the $a_{0,k} \in \bZ_p$ to precision $M = N + n(\delta-1)$. Subsequently, each step from $m-1$ to $m$ decreases both the precision and the valuations by up to $\delta-1$, so that we end at $m = n$ with precision $N$ and valuation $\geq -n(\delta-1)$, still requiring only~$M$ digits.
	\end{proof}
	
	The algorithms outlined above for solving Problems \ref{problem:polylog-series} and \ref{problem:g-series} are implemented in the functions \texttt{compute\_g} and \texttt{compute\_polylog\_series} of the accompanying Sage code \cite{sage_code}. 
	

	\subsection{Finding roots of power series}
	
	Consider the following problem.
	
	\begin{problem}
		\label{problem:root-finding}
		Given a $p$-adic approximation of order~$N$ of a power series $f(t) \in \bQ_p[\![t]\!]$ which converges on~$\bZ_p$, determine the set of roots of $f$ in $\bZ_p$.
	\end{problem}
	
	After rescaling by a power of~$p$, one can assume that the power series has coefficients in $\bZ_p$ and has nonzero reduction mod~$p$. Then, in principle, finding the roots of~$f(t)$ is achieved by Hensel lifting. There are however some subtleties to take into account if we are only given a $p$-adic approximation of $f(t)$. 
	\begin{enumerate}
		\item Knowing $f(t)$ to precision~$N$ might not be enough to decide whether a root modulo $p^N$ lifts to a root in~$\bZ_p$. For example, the polynomials $f_1(t) = t^2 - 1$ and $f_2(t) = t^2 - 5$ over $\bZ_2$ agree modulo $4$ but the roots $\pm 1$ in $\bZ/4\bZ$ only lift to roots in $\bZ_2$ of $f_1$, not of~$f_2$. 
		\item Even if the roots of $f(t)$ modulo~$p^N$ lift to roots in $\bZ_p$, those roots might be determined only up to a lower precision. This happens if the root is not a simple root modulo~$p$. For example, $f_1(t) = t^2 - 1$ and $f_2(t) = t^2-9$ agree modulo~$8$ but their zero sets $\{1,-1\}$ and $\{-3,3\}$ agree only modulo~$4$. 
	\end{enumerate}
	
	The function \texttt{polrootspadic} of PARI/GP (which can also be called from Sage) unfortunately does not take the aforementioned issues with inexact coefficients into account. Therefore, we implemented a function \texttt{Zproots} ourselves, also available at \url{https://github.com/martinluedtke/RefinedCK}, which solves \Cref{problem:root-finding} while taking care of precision questions. For example, the Sage code
	\begin{minted}[frame=lines,baselinestretch=1.2,bgcolor=gray!8,fontsize=\footnotesize]{python}
K = Qp(2,prec=2)
R.<t> = K['t']
Zproots(t^2-1)  # => PrecisionError
	\end{minted}
	results in a \texttt{PrecisionError} since the precision~$2$ is not enough to decide whether the roots $\pm 1$ in $\bZ/4\bZ$ lift to $\bZ_2$. On the other hand, increasing the precision to~$3$, 
	\begin{minted}[frame=lines,baselinestretch=1.2,bgcolor=gray!8,fontsize=\footnotesize]{python}
K = Qp(2,prec=3)
R.<t> = K['t']
Zproots(t^2-1)  # => [1 + O(2^2), 1 + 2 + O(2^2)]
	\end{minted}
	correctly finds $\{1 + O(2^2), 3 + O(2^2)\}$ as the set of roots in $\bZ_2$, those roots being determined modulo~$4$.
	
	The precise version of Hensel's Lemma being used is the following:
	
	\begin{lemma}[{\cite[Theorem~8.2]{konrad:hensel}}]
		\label{thm:hensel}
		Let $f(t) \in \bZ_p[\![t]\!]$ be a power series which converges on $\bZ_p$. Let $a \in \bZ_p$ and set $d \coloneqq v_p(f'(a))$. If $d < v_p(f(a))/2$, then there is a unique $\alpha \in a + p^{d+1} \bZ_p$ such that $f(\alpha) = 0$. Moreover, $v_p(\alpha - a) = v_p(f(a)) - d$.
	\end{lemma}
	
	From this we obtain the following proposition which says exactly to which precision the roots of an inexact power series can be known.
	
	\begin{prop}
		\label{thm:inexact-roots}
		Let $f(t)$ and $\tilde f(t)$ be two power series in $\bZ_p[\![t]\!]$ which converge on $\bZ_p$ and satisfy $f \equiv \tilde f \bmod p^N \bZ_p[\![t]\!]$. Let $a \in \bZ_p$ and set $\tilde d \coloneqq v_p(\tilde f'(a))$. If $\tilde d < N/2$ and $\tilde d < v_p(\tilde f(a))/2$, then $f$ and $\tilde f$ each have a unique root $\alpha$ resp. $\tilde \alpha$ in $a + p^{\tilde d + 1}\bZ_p$, and the roots satisfy $v_p(\alpha - \tilde \alpha) \geq N-\tilde d$.
	\end{prop}
	
	\begin{proof}
		Since $f \equiv \tilde f \bmod  p^N \bZ_p[\![t]\!]$, also $f' \equiv \tilde f' \bmod  p^N \bZ_p[\![t]\!]$ and thus $f'(a) \equiv \tilde f'(a) \bmod p^N$. The valuation $\tilde d$ of $\tilde f'(a)$ is smaller than $N/2 < N$ by assumption, so we get $d\coloneqq v_p(f'(a)) = v_p(\tilde f'(a)) = \tilde d$. Moreover, we have
		\begin{align*}
			v_p(f(a)) &\geq \min(v_p(f(a) - \tilde f(a)), v_p(\tilde f(a))) \geq \min(N, v_p(\tilde f(a))) > 2 \tilde d = 2d,
		\end{align*}
		where both assumptions on~$\tilde d$ are used in the strict inequality. Thus, by \Cref{thm:hensel}, $f$ and $\tilde f$ both have a unique root $\alpha$ resp.\ $\tilde \alpha$ in $a + p^{\tilde d +1}\bZ_p$, proving the first part of the proposition. We now verify the hypotheses of \Cref{thm:hensel} for~$f$ again but with $\tilde \alpha$ in place of~$a$. Since $\tilde \alpha \equiv a \bmod p^{\tilde d+1}$, also $f'(\tilde \alpha) \equiv f'(a) \bmod p^{\tilde d+1}$, thus $v_p(f'(\tilde \alpha)) = v_p(f'(a)) = \tilde d$. Also, we have $f(\tilde \alpha) \equiv \tilde f(\tilde \alpha) = 0 \bmod p^N$, hence $v_p(f(\tilde \alpha)) \geq N > 2 \tilde d$. Now, by \Cref{thm:hensel}, $f$ has a unique root in $\tilde \alpha + p^{d+1} \bZ_p$. But this root must necessarily be $\alpha$. Now the ``moreover'' statement of \Cref{thm:hensel} yields $v_p(\alpha - \tilde \alpha) = v_p(f(\tilde \alpha)) - \tilde d \geq N - \tilde d$.
	\end{proof}
	
	Based on \Cref{thm:inexact-roots}, we obtain an algorithm to resolve \Cref{problem:root-finding}. Assume that a power series $f(t) \in \bZ_p[\![t]\!]$ is given to $p$-adic precision~$N$. Suppose we want to find the roots of $f$ in a $p$-adic disc $a + p^m \bZ_p$ for some $a \in \bZ_p$ and $m \geq 0$, and suppose we know that
	\begin{equation}
		\label{eq:hensel-induction}
		v_p(f'(a)) \geq m-1, \quad v_p(f(a)) \geq 2(m-1).
	\end{equation}
	(At the beginning of the algorithm we take $a = 0$ and $m = 0$ to search in all of $\bZ_p$.) We assume also that $N \geq 2m-1$.  For any $a + p^mb$ in $a + p^m \bZ_p$, using $v_p(f'(a)) \geq m-1$ we have
	\[ f(a + p^m b) \equiv f(a) + p^m b f'(a) \equiv f(a) \bmod p^{2m-1}, \]
	therefore $f(a) \equiv 0 \bmod p^{2m-1}$ is a necessary condition for the existence of roots in $a + p^m \bZ_p$. Since $f$ is known to precision $N$ and we have $N \geq 2m-1$, we can check whether this condition is satisfied. Assume that this is the case. Then we check whether $d \coloneqq v_p(f'(a))$ is equal to $m-1$. If yes, then by Hensel's Lemma, there is a unique root in $a + p^m\bZ_p$, and by \Cref{thm:inexact-roots}, knowing $f$ to precision~$N$ determines this root to precision $N-d$. It is computed by Newton iteration, starting with the value~$a$. If $d \geq m$ on the other hand, the congruence $f(a + p^m b) \equiv f(a)$ holds even modulo $p^{2m}$, so if $f(a) \not \equiv 0 \bmod p^{2m}$ (which requires $N \geq 2m$ to check), then we can conclude that there are no roots in $a + p^m \bZ_p$. Otherwise, we continue to search in the smaller discs $a + ip^m + p^{m+1}\bZ_p$ for $i=0,\ldots,p-1$, which form a partition of $a + p^m \bZ_p$. The conditions~\eqref{eq:hensel-induction} are satisfied for each of these smaller discs where $a$ is replaced by $a + ip^m$ and $m$ is replaced by $m+1$. Assuming that $N \geq 2m+1$, we can proceed as before for each of them. Overall, this leads to the search space $\bZ_p$ being explored in a tree-like manner, recursively subdividing into smaller and smaller discs until we can either decide that they don't contain a root, that they contain exactly one root, or they have become too small relative to the precision~$N$ to detect the roots, in which case the algorithm will fail with a \texttt{PrecisionError}. 
	
	\begin{rem}
		\label{rmk:double-root}
		In finite precision, a power series with a repeated root in $\bZ_p$ cannot be distinguished from a power series which has two roots lying very close to each other. The algorithm will always raise a \texttt{PrecisionError} in this case. If on the other hand $f$ has only simple roots in $\bZ_p$, the outlined algorithm will be able to determine those roots when $f$ is specified with sufficiently large precision. 
	\end{rem}

	\section{Analysis of depth 2 loci}
	\label{sec: sizes}
	
	Using the methods described in §\ref{sec:depth2-loci}, we have computed the Chabauty--Kim loci $X(\bZ_p)_{\{2,q\},2}^{(1,0)}$ for many $p$ and $q$. Complete data can be found on GitHub \cite{sage_code}.
	
	\subsection{Observations}
	
	Taking $q = 3$, \Cref{tbl:depth2-3} shows the size of $X(\bZ_p)_{\{2,3\},2}^{(1,0)}$ for a few choices of auxiliary prime~$p$.
	\begin{table}
	\begin{center}
		\begin{tabular}{c||c|c|c|c|c|c|c|c|c|c|c|c|c|c}
			\hline
			$p$ & 5 & 7 & 11 & 13 & 17 & 19 & 23 & 29 & 31 & \ldots & 1091 & 1093 & 1097 & \ldots \\
			\hline
			$\#X(\bZ_p)_{\{2,3\},2}^{(1,0)}$ & 6 & 8 & 18 & 16 & 22 & 20 & 20 & 26 & 36 & \ldots & 1076 & 2154 & 1078 & \ldots \\
			\hline
		\end{tabular}
		\caption{Sizes of the depth~2 Chabauty--Kim loci $X(\bZ_p)_{\{2,3\},2}^{(1,0)}$ for various choices of auxiliary prime~$p$}
		\label{tbl:depth2-3}
	\end{center}
	\end{table}
	We have computed these loci for all primes $p < 5000$ and made the following observations:
	\begin{enumerate}
		\item The size of the locus is even in each case.
		\item Each residue disc of $X(\bZ_p)$ contains either $0$ or $2$ points of $X(\bZ_p)_{\{2,3\},2}^{(1,0)}$.
		\item For most primes, the locus is roughly of size~$p$, so the $p-2$ residue discs are split more or less evenly between containing two points and containing no points of the locus.
		\item When $p$ is equal to one of the two known Wieferich primes $1093$ and $3511$, the locus is of size~$\approx 2p$. For $p=1093$ only $14$ residue discs contain no points. These are stable under the $S_3$-action and include the residue discs of $i = \sqrt{-1}$ and the primitive $6$-th roots of unity~$\zeta_6$. For $p = 3511$, only~$2$ residue discs contain no points of the locus, namely those of $\zeta_6$ and $\zeta_6^{-1}$.
	\end{enumerate}
	
	Similar observations hold for primes~$q$ other than~$3$. We have computed the size of the depth~2 loci $X(\bZ_p)_{\{2,q\},2}^{(1,0)}$ for all $q < 100$ and $p < 1000$; an excerpt is shown in \Cref{tbl:depth2-q}. 
	\begin{table}
		\begin{center}
			\begin{tabular}{c||c|c|c|c|c|c|c|c|c|c|c|c|c|c}
				\hline
				$p$ & 3 & 5 & 7 & 11 & 13 & 17 & 19 & 23 & 29 & 31 & 37 & 41 & 43 & 47 \\
				\hline
				$\#X(\bZ_p)_{\{2,5\},2}^{(1,0)}$ & 3 & -- & 6 & 10 & 8 & 12 & 18 & 24 & 38 & 28 & 36 & 34 & 50 & 44 \\
				$\#X(\bZ_p)_{\{2,7\},2}^{(1,0)}$ & 3 & 6 & -- & 10 & 12 & 15 & 18 & 22 & 32 &
				 38 & 36 & 44 & 38 & 44 \\
				 $\#X(\bZ_p)_{\{2,11\},2}^{(1,0)}$ & 3 & 6 & 6 & -- & 8 & 14 & 18 & 18 & 38 & 28 & 40 & 34 & 36 & 48 \\
				 $\#X(\bZ_p)_{\{2,19\},2}^{(1,0)}$ & 2 & 4 & 8 & 8 & 18 & 20 & -- & 20 & 30 & 26 & 36 & 44 & 78 & 44 \\
				 
				 \hline
			\end{tabular}
			\caption{Sizes of the depth~2 Chabauty--Kim loci $X(\bZ_p)_{\{2,q\},2}^{(1,0)}$ for $q = 5, 7, 11, 19$ and $p < 50$}
			\label{tbl:depth2-q}
		\end{center}
	\end{table}
	For $p = 3$, the locus always has size~$2$ or $3$. This is a general fact proved in \cite[Prop.~3.14]{BBKLMQSX}. For $p \geq 5$, most of the time the size of the locus is even, but it can occasionally be odd; for example the size is~$15$ for $p = 17$ and $q = 7$. Also, the size is usually close to~$p$, but in a few cases it is close to~$2p$. The latter always occurs if $p$ is one of the known Wieferich primes $1093$ and $3511$, but occasionally it happens for non-Wieferich~$p$ as well; for example, \Cref{tbl:depth2-q} shows that the locus for $q = 19$ and $p = 43$ has size~$78$. 
	
	\subsection{Newton polygon analysis}
	
	We can explain many of these observations by analysing the Newton polygons of the power series defining the Chabauty--Kim locus $X(\bZ_p)_{\{2,q\},2}^{(1,0)}$. Setting $a \coloneqq \frac{a_{\tau_q \tau_2}}{a_{\tau_2} a_{\tau_q}}$, this locus is defined in $X(\bZ_p)$ by the function
	\begin{equation}
		\label{eq: function}
		f(z) \coloneqq \Li_2(z) - a\log(z) \Li_1(z).
	\end{equation}
	
	\begin{lemma}
		\label{thm:coeffs}
		Let $\zeta \neq 1$ be a $(p-1)$-st root of unity in $\bZ_p$. Then the first three coefficients of the power series $f(\zeta + pt) = \sum_{k=0}^\infty c_k t^k \in \bQ_p[\![t]\!]$ are given by
		\begin{align*}
			c_0 &= \Li_2(\zeta),\\
			c_1 &= p(1-a) \frac{\Li_1(\zeta)}{\zeta},\\
			c_2 &= p^2(1-2a)\frac1{2\zeta(1 - \zeta)} - p^2 (1-a) \frac{\Li_1(\zeta)}{2\zeta^2}.
		\end{align*}
	\end{lemma}
	
	\begin{proof}
		The $k$-th coefficient of $f(\zeta + pt)$ is given by $p^k f^{(k)}(\zeta)/k!$. Computing the first two derivatives of $f(z)$ is straightforward using the differential equations $\Li_2'(z) = \Li_1(z)/z$ and $\Li_1'(z) = 1/(1-z)$. Plugging in $\zeta$ and using $\log(\zeta) = 0$ gives the claimed formulas.
	\end{proof}

	\begin{prop}
		\label{thm:depth2-size}
		Let $q$ be an odd prime and let $p \geq 5$ be a prime not equal to~$q$. Set $a \coloneqq \frac{a_{\tau_q \tau_2}}{a_{\tau_2} a_{\tau_q}}$ and assume that $v_p(a) < 0$ or $a \not\equiv \frac12 \bmod p$.	Then the Chabauty--Kim set $X(\bZ_p)_{\{2,q\},2}^{(1,0)}$ contains at most~$2$ points from each residue disc. 
	\end{prop}
	
	\begin{proof}
		Let $\nu \coloneqq v_p(a)$ and assume first that $\nu = v_p(a) < 0$. Let $\zeta \neq 1$ be a $(p-1)$-st root of unity in $\bZ_p$. By \Cref{thm:coeffs}, the valuations of the first three coefficients of $f(\zeta + pt)$ are $\geq 2$, $\geq 2+\nu$, $= 2+\nu$, respectively. It follows from \Cref{thm:polylog-series-estimates} that the $k$-th coefficients of both $\Li_2(\zeta+pt)$ and $\log(\zeta+pt)\Li_1(\zeta+pt)$ have valuation $\geq k - 2 \lfloor \log_p(k) \rfloor$. 
		For $k \geq 3$ this is always $\geq 3$,
		so that all coefficients of $f(\zeta + pt)$ for $k \geq 3$ have valuation $\geq 3+\nu$. In conclusion, the minimal valuation of all coefficients is equal to $2+\nu$; it is attained at the $t^2$-coefficient and this is the last time it is attained. By Strassmann's Theorem it follows that $f(z)$ has at most~$2$ zeros on $U_{\zeta}$.
		
		Assume now that $\nu = v_p(a) \geq 0$ and $a \not\equiv \frac12 \bmod p$. In this case the first three coefficients of $f(\zeta+pt)$ have valuations $\geq 2$, $\geq 2$, $=2$, respectively, and all other coefficients have larger valuation. We conclude again by Strassmann's Theorem.
	\end{proof}
	
	\begin{rem}
		From \cite[Lemma~7.0.5]{betts:effective} one has an a priori bound of
		\[ \#X(\bZ_p)_{\{2,q\},2}^{(1,0)} \leq 8(p-2) + \frac{8(p-1)}{\log(p)} \]
		for all $p \neq q$. When \Cref{thm:depth2-size} applies, we get the improved bound
		\[ \#X(\bZ_p)_{\{2,q\},2}^{(1,0)} \leq 2(p-2) \]
		since there are $p-2$ residue discs, each containing at most~2 points of the locus. This bound is sometimes attained, e.g., $\#X(\bZ_{29})_{\{2,41\},2}^{(1,0)} = 54 = 2\cdot(29-2)$.
	\end{rem}
	
	In \Cref{thm:depth2-size}, the valuation of $a = \frac{a_{\tau_q \tau_2}}{a_{\tau_2} a_{\tau_q}}$ is determined by the valuation of $a_{\tau_q \tau_2}$ and the valuations of the $p$-adic logarithms $a_{\tau_2} = \log(2)$ and $a_{\tau_q} = \log(q)$. About the first we can show the following:
	
	\begin{lemma}
		\label{thm:aq2-valuation}
		For fixed~$q$, we have $v_p(a_{\tau_q \tau_2}) \geq 2$ for all but finitely many~$p$.
	\end{lemma}
	
	\begin{proof}
		Recall from §\ref{sec:DCW-coeffs} that there are rational numbers $c_i \in \bQ$ and $t_i \in \bQ \smallsetminus \{0,1\}$ (independent of~$p$) such that $a_{\tau_q \tau_2} = -\sum_i c_i \Li_2(t_i).$
		Analysing the coefficients of the power series of $\Li_2(\zeta + pt)$ as in the proof of \Cref{thm:depth2-size}, one sees that for $p \geq 3$ all coefficients have valuation $\geq 2$, so that $v_p(\Li_2(x)) \geq 2$ for all $x \in \bZ_p$ with $x \not\equiv 0,1\bmod p$. In particular, if $p$ is not any of the finitely many primes occuring in the prime factorisation of $c_i$, $t_i$ or $1-t_i$, then $v_p(a_{\tau_q \tau_2}) \geq 2$.
	\end{proof}
	
	The valuation of a $p$-adic logarithm is related to the Wieferich property. Recall that $p$ is called a \emph{base-$b$ Wieferich prime} if $b^{p-1} \equiv 1 \bmod p^2$. A base-$2$ Wieferich prime is simply called a Wieferich prime.
	
	\begin{lemma}
		\label{thm:wieferich}
		Let $p$ be an odd prime and let $b > 1$ be an integer not divisible by~$p$. The valuation of the $p$-adic logarithm $\log(b)$ is given by
		$v_p(\log(b)) = v_p(b^{p-1} - 1)$. 
		In particular, $v_p(\log(b)) > 1$ if and only if~$p$ is a base-$b$ Wieferich prime.
	\end{lemma}
	
	\begin{proof}
		We have $\log(b) = \frac1{p-1} \log(b^{p-1})$, hence the $p$-adic valuation of $\log(b)$ agrees with the valuation of $\log(b^{p-1})$. By Fermat's little theorem, $b^{p-1} \equiv 1 \bmod p$. For $x \in 1 + p\bZ_p$ one has $v_p(\log(x)) = v_p(x-1)$. The claim follows.
	\end{proof}
	
	It is conjectured that infinitely many Wieferich primes exist, although they occur very scarcely. Heuristically, the number of Wieferich primes below~$x$ grows like $\log(\log(x))$. The only currently known Wieferich primes are $1093$ and $3511$. Interestingly, it is also not known whether there are infinitely many \emph{non}-Wieferich primes, although this would follow from the abc conjecture \cite{silverman:wieferich}.
	
	\subsection{Large loci}
	
	We can explain the observation that $\#X(\bZ_p)_{\{2,q\},2}^{(1,0)} \approx 2p$ when $p$ is a base-2 or base-$q$ Wieferich prime. Usually the inequality $v_p(a_{\tau_q \tau_2}) \geq 2$ from \Cref{thm:aq2-valuation} is an equality. In this case, the Wieferich property and \Cref{thm:wieferich} imply that the valuation~$\nu$ of $a = \frac{a_{\tau_q \tau_2}}{a_{\tau_2} a_{\tau_q}}$ is negative. For most $\zeta$ we have $v_p(\Li_1(\zeta)) = 1$ (the primitive $6$-th roots of unity being an exception). Then the first three coefficients of $f(\zeta + pt)$ from \Cref{thm:coeffs} have valuations $\geq2$, $=2+\nu$, $=2+\nu$, respectively, and all subsequent coefficients have larger valuation. The Newton polygon has then exactly two segments of non-positive slope: negative slope from 0 to 1 and slope zero from 1 to 2. It follows that $f(z)$ has exactly two roots in the residue disc $U_{\zeta}$. With most of the $p-2$ residue discs containing~$2$ points, the Chabauty--Kim locus contains roughly $2p$ points in total.
	
	This explains why the Chabauty--Kim loci are exceptionally large when $p = 1093$ or $p = 3511$ since these are base-2 Wieferich primes. It also explains why the locus for $q = 19$ and $p = 43$ is exceptionally large since $43$ is a base-$19$ Wieferich prime. Another large locus (of size~$42$) occurs for $q = 47$ and $p = 23$. This one is not explained by a Wieferich property but rather the fact that the valuation $v_{23}(a_{\tau_{47} \tau_2}) = 1$ is smaller than expected. The prime $p = 23$ belongs to the finitely many exceptions in \Cref{thm:aq2-valuation}, causing $\nu = v_p(a) = -1$ to be negative.
	
	\subsection{Typical loci}
	We can also explain heuristically why $\# X(\bZ_p)_{\{2,q\}}^{(1,0)} \approx p$ almost always when $p$ is not a base-2 or base-$q$ Wieferich prime. Typically, $a_{\tau_q \tau_2}$ has valuation~2, so that $a = \frac{a_{\tau_q \tau_2}}{a_{\tau_2} a_{\tau_q}}$ has valuation~$\nu = 0$. Consider first the case that $a \not\equiv \frac12 \bmod p$. For most $\zeta$ we have $v_p(\Li_2(\zeta)) = 2$. On the residue discs of such~$\zeta$, the valuations of the first three coefficients of $f(\zeta + pt)$ given in \Cref{thm:coeffs} are then $=2, \geq 2, =2$, respectively, and all subsequent coefficients have valuation~$\geq 3$. After normalising, the power series reduces to a polynomial in $\bF_p[t]$ of degree~$2$ with nonzero constant coefficient. If its discriminant behaves like a random element in $\bF_p$, the cases of having two simple roots or no roots in $\bF_p$ should occur about half of the time each, with a small leftover probability of $1/p$ of having a double root. 
	By Hensel's lemma, we can therefore expect the $p-2$ residue discs to be split roughly evenly between containing $0$ points and $2$ points of the Chabauty--Kim locus, amounting to $\approx p$ points in total.
	
	Assume now that $a \equiv \frac12 \bmod p$. In this case, using that $v_p(\Li_1(\zeta)) = 1$ for most $(p-1)$-st roots of unity~$\zeta$ (primitive 6-th roots of unity being an exception), the Newton polygon of $f(\zeta + pt)$ usually has just a single non-positive slope, so that almost all residue discs contain exactly~$1$ point of the locus. The total size is then $\approx p$ as well, despite distributing differently over the residue discs than in the case $a \not\equiv \frac12 \bmod p$.

	\section{Verifying Kim's Conjecture}
	\label{sec:verifying-kim}
	
	We now verify instances of \Cref{kim-conjecture-refined} for the thrice-punctured line over $\bZ[1/6]$ in depth~$4$, saying that the refined Chabauty--Kim locus $X(\bZ_p)_{\{2,3\},4}^{\min}$ consists exactly of the 21 points 
	\[ X(\bZ[1/6]) = \left\{2, \frac12, -1, 3, \frac13, \frac23, \frac32, -\frac12,-2, 4, \frac14, \frac34, \frac43, -\frac13, -3, 9, \frac19, \frac89, \frac98, -\frac18, -8 \right\}. \]
	(Finding $X(\bZ[1/6])$ boils down to finding all pairs of consecutive integers containing only the prime factors~2 and~3. The fact that $(1,2), (2,3), (3,4), (8,9)$ is a complete list of such pairs was proved in 1342 by Gersonides, also known as Levi ben Gershon. The proof was published in his book \emph{The Harmony of Numbers}.)
	
	Recall from \Cref{thm:kim-conjecture-reduction} that it suffices to verify Conjecture \ref{kim-conjecture-sigma} for the polylogarithmic depth~$4$ quotient of the fundamental group and for the two refinement conditions $\Sigma = (1,1)$ and $\Sigma = (1,0)$. Recall also that the conjecture for $\Sigma = (1,1)$ was already checked for all $p < 10^5$ in \cite[Remark~3.6]{BBKLMQSX} (and is proved in higher depth for arbitrary~$p$, see \Cref{rem:11-locus}.) Therefore it is enough to look at the case $\Sigma = (1,0)$ and show that the inclusion
	\[ \{-3,-1,3,9\} = X(\bZ[1/6])_{(1,0)} \subseteq X(\bZ_p)_{\{2,3\},\PL,4}^{(1,0)} \]
	is an equality. As shown in §\ref{sec:kim-functions}, the locus in question is defined by the two equations
	\begin{align}
		\label{eq:f2}
		&f_2(z) \coloneqq \log(2) \log(3) \Li_2(z) + \Li_2(3) \log(z) \Li_1(z) = 0,\\
		&f_4(z) \coloneqq  \det \begin{pmatrix}
			\Li_4(z) & \log(z) \Li_3(z) & \log(z)^3 \Li_1(z) \\
			\Li_4(3) & \log(3) \Li_3(3) & \log(3)^3 \Li_1(3) \\
			\Li_4(9) & \log(9) \Li_3(9) & \log(9)^3 \Li_1(9) 
		\end{pmatrix} = 0. 
		\label{eq:f4}
	\end{align}
	
	The first equation~\eqref{eq:f2} defines the depth~2 Chabauty--Kim locus $X(\bZ_p)_{\{2,3\},2}^{(1,0)}$ and we explained in §\ref{sec:depth2-loci} how it can be computed. In order to verify Kim's Conjecture in depth~4 it suffices to check that $f_4(z) \neq 0$ for every point~$z$ of the depth~2 locus which is not in $\{-3,-1,3,9\}$.
	
	We demonstrate this for $p = 5$, continuing the example from §\ref{sec: depth2-example}. There we found that the depth~2 locus contains two points in addition to $\{-3,-1,3,9\}$: one of them is $z = 2$ and the other is given in~\eqref{eq:extra-point}.
	In Sage, we can define the depth 4 function $f_4$ from Eq.~\eqref{eq:f4} as follows:
	\begin{minted}[frame=lines,baselinestretch=1.2,bgcolor=gray!8,
		fontsize=\footnotesize]{python}
# p-adic polylogarithms
K = Qp(5)
log = lambda z: K(z).log()
Li = [lambda z,n=n: K(z).polylog(n) for n in range(5)]

# our depth-4 function
def f(z):
    rows = [[Li[4](x), log(x)*Li[3](x), log(x)^3*Li[1](x)] for x in [z,3,9]]
    return matrix(rows).determinant()
	\end{minted}
	We can check that the depth~4 function does indeed vanish on $-3,-1,3,9$:
	\begin{minted}[frame=lines,baselinestretch=1.2,bgcolor=gray!8,
		fontsize=\footnotesize]{python}
f(-3)  # => O(5^20)
f(-1)  # => O(5^28)
f(3)  # => O(5^20)
f(9)  # => O(5^20)
	\end{minted}
	Now we verify that the function does \emph{not} vanish on the extra points~$2$ and $z_0$:
\begin{minted}[frame=lines,baselinestretch=1.2,bgcolor=gray!8,fontsize=\footnotesize]{python}
f(2)
# => 4*5^13 + 4*5^14 + 3*5^15 + 5^16 + 3*5^18 + 3*5^19 + O(5^20)
f(3 + 5^2 + 2*5^3 + 5^4 + 3*5^5 + 5^6 + 5^7 + 5^9 + 2*5^10 + 3*5^11 + 2*5^12 + O(5^13))
# => 4*5^13 + O(5^14)
	\end{minted}
	This shows that \Cref{kim-conjecture-refined} holds for $S = \{2,3\}$ and $p = 5$. 
	
	The accompanying Sage code has a function
	\begin{center}
		\texttt{CK\_depth\_4\_locus(p, q, N, coeffs)}
	\end{center}
	which computes an approximation of the locus $X(\bZ_p)_{\{2,q\},\PL,4}^{(1,0)}$ for any $p$ and $q$. It takes the tuple of $p$-adic coefficients $(a_{\tau_q \tau_2}, a, b, c)$ as an argument, where $a_{\tau_q \tau_2}$ is the DCW coefficient appearing in the depth~2 function~\eqref{eq:depth2-function-intro}, and $(a,b,c)$ are the coefficients in the depth~4 function \eqref{eq:depth4-equation-intro}. At the moment we only know all these constants explicitly in the case $q = 3$ where $a_{\tau_3 \tau_2} = -\Li_2(3)$ and for $(a,b,c)$ one can either use the formulas from \Cref{thm: periods} or $2 \times 2$-minors of the matrix in \eqref{eq:f4}. The code is however flexible enough to compute the Chabauty--Kim loci for other primes~$q$ in the future. The function \texttt{Z\_one\_sixth\_coeffs(p,N)} can be used to compute the $p$-adic coefficients for $q = 3$ conveniently. Let us use this to determine the depth~4 locus for $p = 7$:
	\begin{minted}[frame=lines,baselinestretch=1.2,bgcolor=gray!8,
		fontsize=\footnotesize]{python}
p = 7; q = 3; N = 10
coeffs = Z_one_sixth_coeffs(p,N)
CK_depth_4_locus(p,q,N,coeffs)
\end{minted}
This outputs the following list of $7$-adic numbers:
\begin{minted}[frame=lines,baselinestretch=1.2,bgcolor=gray!8,
	fontsize=\footnotesize]{python}
[2 + 7 + O(7^9),
 3 + O(7^9),
 4 + 6*7 + 6*7^2 + 6*7^3 + 6*7^4 + 6*7^5 + 6*7^6 + 6*7^7 + 6*7^8 + O(7^9),
 6 + 6*7 + 6*7^2 + 6*7^3 + 6*7^4 + 6*7^5 + 6*7^6 + 6*7^7 + O(7^8)]
	\end{minted}
	These are precisely the $\{2,3\}$-integral points $9, 3, -3, -1$, which shows that Kim's Conjecture for $S = \{2,3\}$ also holds with the auxiliary prime~$p = 7$. 
	
	\begin{rem}
		\label{rem: common-zeros-inexact}
		Computing the common zero set of two inexact functions is not a well-posed problem. For the locus defined by $f_2(z) = f_4(z) = 0$ we can only compute approximations of the roots of $f_2$ and check whether $f_4(z)$ is \emph{indistinguishable} from~$0$ up to the given precision. If the precision is chosen too low one might not be able to rule out certain points to be roots of~$f_4$ and end up with a too large set. However, in all cases we considered, increasing the precision if necessary, we were always able to eliminate all points other than the four points $\{-3,-1,3,9\}$ which we know to be common zeros of~$f_2$ and~$f_4$.
	\end{rem}
	
	Computing depth~4 loci by the methods layed out in this paper we have verified:
	
	\begin{thm}
		Kim's Conjecture for $S = \{2,3\}$ and refinement condition $\Sigma = (1,0)$ holds for the polylogarithmic depth~4 quotient, i.e., the inclusion
		\[ \{-3,-1,3,9\} = X(\bZ[1/6])_{(1,0)} \subseteq X(\bZ_p)_{\{2,3\},\PL,4}^{(1,0)} \]
		is an equality, for all auxiliary primes~$p$ with $5\leq p < 10{,}000$.
	\end{thm}
	
	\begin{cor}
		\label{thm:main}
		Conjecture~\ref{kim-conjecture-refined} for $S = \{2,3\}$ holds in depth~$4$ for all auxiliary primes~$p$ with $5 \leq p < 10{,}000$.
	\end{cor}

	\printbibliography
\end{document}